\documentclass[a4paper,10pt]{article}

\newtheorem{theorem}{Theorem}[section]
\newtheorem{lemma}[theorem]{Lemma}

\newtheorem{corollary}[theorem]{Corollary}

\newenvironment{proof}[1][Proof]{\begin{trivlist}
\item[\hskip \labelsep {\bfseries #1}]}{\end{trivlist}}

\newenvironment{example}[1][Example]{\begin{trivlist}
\item[\hskip \labelsep {\bfseries #1}]}{\end{trivlist}}

\newcommand{\qed}{\nobreak \ifvmode \relax \else      \ifdim\lastskip<1.5em \hskip-\lastskip
\hskip1.5em plus0em minus0.5em \fi \nobreak
\vrule height0.75em width0.5em depth0.25em\fi}

\usepackage{amssymb}
\usepackage{mathrsfs}

\begin{document}

\setlength{\unitlength}{0.1mm}

\newcommand{\balpha}{ \alpha^{\triangleleft} }

\newcommand{\bbeta}{ \beta^{\triangleleft}}

\newcommand{\lroof}{\mbox{\begin{picture}(3,0)
                                \put(0,0){\line(0,1){10}}

                                \put(0,10){\line(1,0){3}}

                           \end{picture}}}

\newcommand{\rroof}{\mbox{\begin{picture}(3,0)
                                \put(3,0){\line(0,1){10}}

                                \put(3,10){\line(-1,0){3}}

                           \end{picture}}}

\newcommand{\longlroof}{\mbox{\begin{picture}(3,0)
                                \put(0,-4){\line(0,1){14}}

                                \put(0,10){\line(1,0){3}}

                           \end{picture}}}

\newcommand{\longrroof}{\mbox{\begin{picture}(3,0)
                                \put(3,-4){\line(0,1){14}}

                                \put(3,10){\line(-1,0){3}}

                           \end{picture}}}

\newcommand{\wedgeline}{\mbox{\begin{picture}(0,0)
                                \put(-1.5,0){$\wedge$}
                                \put(-2,3){\line(0,-1){3}}
                              \end{picture}}}
\newcommand{\Boox}{\mbox{\begin{picture}(0,0)
                                \put(0,-3){ \framebox(10,10){\tiny{$j-1$}} }
                            \end{picture}}}

\newcommand{\hooka}{\mbox{\begin{picture}(0,0)
                                \put(-5,-3){ \framebox(10,10){} }
                                \put(-5,-13){ \framebox(10,10){} }
                                \put(-5,-23){ \framebox(10,10){} }
                                \put(-5,-33){ \framebox(10,10){} }
                                \put(-5,-43){ \framebox(10,10){} }
                                \put(-5,-53){ \framebox(10,10){} }
                            \end{picture}}}

\newcommand{\hookb}{\mbox{\begin{picture}(0,0)
                                \put(-5,-3){ \framebox(10,10){} }
                                \put(5,-3){ \framebox(10,10){} }
                                \put(-5,-13){ \framebox(10,10){} }
                                \put(-5,-23){ \framebox(10,10){} }
                                \put(-5,-33){ \framebox(10,10){} }
                                \put(-5,-43){ \framebox(10,10){} }
                            \end{picture}}}

\newcommand{\hookc}{\mbox{\begin{picture}(0,0)
                                \put(-5,-3){ \framebox(10,10){} }
                                \put(5,-3){ \framebox(10,10){} }
                                \put(15,-3){ \framebox(10,10){} }
                                \put(-5,-13){ \framebox(10,10){} }
                                \put(-5,-23){ \framebox(10,10){} }
                                \put(-5,-33){ \framebox(10,10){} }
                            \end{picture}}}

\newcommand{\hookd}{\mbox{\begin{picture}(0,0)
                                \put(-5,-3){ \framebox(10,10){} }
                                \put(5,-3){ \framebox(10,10){} }
                                \put(15,-3){ \framebox(10,10){} }
                                \put(25,-3){ \framebox(10,10){} }
                                \put(-5,-13){ \framebox(10,10){} }
                                \put(-5,-23){ \framebox(10,10){} }
                            \end{picture}}}

\newcommand{\hooke}{\mbox{\begin{picture}(0,0)
                                \put(-5,-3){ \framebox(10,10){} }
                                \put(5,-3){ \framebox(10,10){} }
                                \put(15,-3){ \framebox(10,10){} }
                                \put(25,-3){ \framebox(10,10){} }
                                \put(35,-3){ \framebox(10,10){} }
                                \put(-5,-13){ \framebox(10,10){} }
                            \end{picture}}}

\newcommand{\hookf}{\mbox{\begin{picture}(0,0)
                                \put(-5,-3){ \framebox(10,10){} }
                                \put(5,-3){ \framebox(10,10){} }
                                \put(15,-3){ \framebox(10,10){} }
                                \put(25,-3){ \framebox(10,10){} }
                                \put(35,-3){ \framebox(10,10){} }
                                \put(45,-3){ \framebox(10,10){} }
                            \end{picture}}}

\setlength{\unitlength}{0.1mm}

\newcommand{\starbox}{\mbox{\begin{picture}(0,0)
                                \put(0,0){ \framebox(10,10){ }}
                                \put(6,2){$*$}
                           \end{picture}}}

\newcommand{\bbox}{\mbox{\begin{picture}(0,0)
                                \put(0,0){ \framebox(10,10){ }}
                                \put(6,2){$b$}
                           \end{picture}}}

\newcommand{\cbox}{\mbox{\begin{picture}(0,0)
                                \put(0,0){ \framebox(10,10){ }}
                                \put(6,2){$c$}
                           \end{picture}}}

\newcommand{\Done}{\mbox{\begin{picture}(0,0)
                                \put(-20,-23){ \framebox(10,10){ }}
                                \put(-10,-13){ \framebox(10,10){ }}
                                \put(-10,-3){ \framebox(10,10){ }}
                                \put(-10,-23){ \framebox(10,10){ }}
                            \end{picture}}}

\newcommand{\Dtwo}{\mbox{\begin{picture}(0,0)
                                \put(-20,-23){ \framebox(10,10){ }}
                                \put(-10,-23){ \framebox(10,10){ }}
                                \put(-20,-13){ \framebox(10,10){ }}
                                \put(-10,-13){ \framebox(10,10){ }}
                                \put(0,-13){ \framebox(10,10){ }}
                                \put(10,-13){ \framebox(10,10){ }}
                                \put(-20,-3){ \framebox(10,10){ }}
                                \put(-10,-3){ \framebox(10,10){ }}
                                \put(0,-3){ \framebox(10,10){ }}
                                \put(10,-3){ \framebox(10,10){ }}
                             \end{picture}}}

\newcommand{\Donedottwo}{\mbox{\begin{picture}(0,0)
                                \put(-20,-23){ \framebox(10,10){ }}
                                \put(-10,-23){ \framebox(10,10){ }}
                                \put(-20,-13){ \framebox(10,10){ }}
                                \put(-10,-13){ \framebox(10,10){ }}
                                \put(0,-13){ \framebox(10,10){ }}
                                \put(10,-13){ \framebox(10,10){ }}
                                \put(-20,-3){ \framebox(10,10){ }}
                                \put(-10,-3){ \framebox(10,10){ }}
                                \put(0,-3){ \framebox(10,10){ }}
                                \put(10,-3){ \framebox(10,10){ }}

                                \put(-30,-53){ \framebox(10,10){ }}
                                \put(-20,-43){ \framebox(10,10){ }}
                                \put(-20,-33){ \framebox(10,10){ }}
                                \put(-20,-53){ \framebox(10,10){ }}
                             \end{picture}}}

\newcommand{\Doneodottwo}{\mbox{\begin{picture}(0,0)
                                \put(-20,-23){ \framebox(10,10){ }}
                                \put(-10,-23){ \framebox(10,10){ }}
                                \put(-20,-13){ \framebox(10,10){ }}
                                \put(-10,-13){ \framebox(10,10){ }}
                                \put(0,-13){ \framebox(10,10){ }}
                                \put(10,-13){ \framebox(10,10){ }}
                                \put(-20,-3){ \framebox(10,10){ }}
                                \put(-10,-3){ \framebox(10,10){ }}
                                \put(0,-3){ \framebox(10,10){ }}
                                \put(10,-3){ \framebox(10,10){ }}

                                \put(-40,-43){ \framebox(10,10){ }}
                                \put(-30,-33){ \framebox(10,10){ }}
                                \put(-30,-23){ \framebox(10,10){ }}
                                \put(-30,-43){ \framebox(10,10){ }}
                             \end{picture}}}

\begin{center}

{\huge Skew Schur Functions of Sums of Fat Staircases}

\

{\LARGE Matthew Morin}

University of British Columbia

\

{Mathematics Subject Classification: 05E05}

\end{center}

\

\begin{center}
\textbf{Abstract}
\end{center}

We define a fat staircase to be a Ferrers diagram corresponding to a partition of the form $(n^{\alpha_n}, {n-1}^{\alpha_{n-1}},\ldots, 1^{\alpha_1})$, where $\alpha = (\alpha_1,\ldots,\alpha_n)$ is a composition, or the $180^\circ$ rotation of such a diagram. 
If a diagram's skew Schur function is a linear combination of Schur functions of fat staircases, we call the diagram a sum of fat staircases. 
We prove a Schur-positivity result that is obtained each time we augment a sum of fat staircases with a skew diagram. 
We also determine conditions on which diagrams can be sums of fat staircases, including necessary and sufficient conditions in the special case when the diagram is a fat staircase skew a single row or column.

\section{Introduction}

The Schur functions are perhaps best known as a basis of the ring of symmetric functions. 
As such, their structure constants $c_{\mu \nu}^{\lambda}$, commonly known as the \textit{Littlewood-Richardson coefficients}, which describe the products
\[ s_\mu s_\nu = \sum_{\lambda} c_{\mu \nu}^{\lambda} s_{\lambda}, \]
are of paramount importance to the structure of the ring. 
However, the development and study of Schur functions has not only demonstrated their importance in the study of symmetric functions, but also in several other areas of mathematics.

The Schur functions appear in the study of representation theorey. 
For instance, in the representations of the symmetric group, the Specht modules are in one-to-one correspondence with the Schur functions. 
Further, given two Specht modules $S^\mu$ and $S^\nu$ we have
\[ (S^\mu \otimes S^\nu)\uparrow^{S_n} = \bigoplus_{\lambda} c_{\mu \nu}^{\lambda} S^{\lambda}. \]

Schur functions are also intrinsic to the structure of the cohomology ring of the Grassmannian, where the Schubert classes are in correspondence to the Schur functions and the cup product of each pair $\sigma_\mu$, $\sigma_\nu$ of Schubert classes satisfies
\[ \sigma_\mu \cup \sigma_\nu = \sum_{\lambda} c_{\mu \nu}^{\lambda} \sigma_\lambda.\]

The ubiquitous Littlewood-Richardson coefficients are known to satisfy $c_{\mu \nu}^{\lambda} \geq 0.$ 
Thus each product $s_\mu s_\nu$ gives rise to a linear combination of Schur functions with non-negative coefficients. 
We call such expressions \textit{Schur-positive}.

In recent years, there has been significant interest in determining instances of Schur-positivity in expressions of the form
\[ s_{\mu} s_{\nu} - s_{\lambda} s_{\rho} \textrm{ } \textrm{ and } \textrm{ } s_{\lambda / \mu} - s_{\rho / \nu}.\]
A collection of work in this vein includes \cite{{g2},{complementcite},{lpp},{m},{mvw2}}.  
Each of these Schur-positive differences gives a set of inequalities that the corresponding Littlewood-Richardson coefficients must satisfy. 
In this paper we shall construct certain families of skew Schur functions and shall obtain certain Schur-positive differences.

\section{Preliminaries}

We begin by briefly introducing the various objects, notations, and results that this paper requires.  
A complete study can be found in sources such as \cite{sagan} or \cite{stanley}.

A \textit{partition} $\lambda$\label{def:lambda} of a positive integer $n$, written $\lambda \vdash n$, is a sequence of weakly decreasing positive integers $\lambda = (\lambda_1,\lambda_2, \ldots, \lambda_k)$ with $\sum_{i=1}^{k} \lambda_i = n$. 
We call each $\lambda_i$ a \textit{part} of $\lambda$, and if $\lambda$ has exactly $k$ parts we say $\lambda$ is of \textit{length} $k$ and write $l(\lambda)=k$.
When $\lambda \vdash n$ we will also write $| \lambda |=n$ and say that the \textit{size} of $\lambda$ is $n$. 

We shall use $j^r$ to denote the sequence $j,j,\ldots, j$ consisting of $r$ $j$'s.
Under this notation, we shall write $\lambda = (k^{r_k}, {k-1}^{r_{k-1}}, \ldots ,1^{r_1})$ \label{parts} for the partition which has $r_1$ parts of size one, $r_2$ parts of size two, \ldots,  and $r_k$ parts of size $k$.

We say $\alpha=(\alpha_1,\alpha_2,\ldots,\alpha_k)$ is a \label{comp} \textit{composition} of $n$ if each $\alpha_i$ is a positive integer and $\sum_{i=1}^{k} \alpha_i = n$.
As with partitions, we call each $\alpha_i$ a \textit{part} of $\alpha$, write $|\alpha|=n$ for the \textit{size} of $\alpha$, and if $\alpha$ has exactly $k$ parts we say $\alpha$ is of \textit{length} $k$ and write $l(\alpha)=k$.
If we relax the conditions to consider sums of non-negative integers, that is, allowing some of the $\alpha_i$ to be zero, then we obtain the concept of a \textit{weak composition} of $n$.

We may sometimes find it useful to treat partitions and weak compositions as vectors with non-negative integer entries.
When we write the vector ${z}=(z_1,z_2, z_3,\ldots, z_n)$ we shall mean the infinite vector \[{z}=(z_1,z_2, z_3,\ldots, z_n, 0, 0, 0, \ldots).\]
We shall only consider vectors with finitely many non-zero entries. 
Hence we shall only display vectors with finite length. 
In this manner we may unambiguously add vectors of different lengths. 
Thus, we have defined addition among partitions and weak-compositions. 
Further, given a positive integer $i$, we shall let $e_i$ denote the $i$-th standard basis vector.
That is, the vector that has its $i$-th entry equal to $1$ and all remaining entries equal to $0$.

\

Given a partition $\lambda$, we can represent it via the diagram of left-justified rows of boxes whose $i$-th row contains $\lambda_i$ boxes. 
The diagrams of these type are called \textit{Ferrers diagrams}, or just \textit{diagrams} for short. 
We shall use the symbol $\lambda$ when refering to both the partition and its Ferrers diagram.

Whenever we find a diagram $D'$ contained in a diagram $D$ as a subset of boxes, we say that $D'$ is a \textit{subdiagram} of $D$.
Suppose partitions $\lambda=(\lambda_1,\lambda_2,\ldots, \lambda_j) \vdash n$ and $\mu=(\mu_1,\mu_2,\ldots, \mu_k) \vdash m$, with $m \leq n$, $k \leq j$, and $\mu_i \leq \lambda_i$ for each $i = 1, 2, \ldots, k$ are given. 
Then $\mu$ is a subdiagram of $\lambda$, and a particular copy of $\mu$ is found at the top-left corner of $\lambda$.  
We can form the \textit{skew \label{skew} diagram} $\lambda / \mu$ by removing that copy of $\mu$ from $\lambda$. 
Henceforth, when we say that $D$ is a \textit{diagram}, it is assumed that $D$ is either the Ferrers diagram of some partition $\lambda$, or $D$ is the skew diagram $\lambda / \mu$ for some partitions $\lambda, \mu$. 

\begin{example}
Here we consider the partitions $\lambda=(4,4,2)$ and $\mu=(3,1)$, and form the skew diagram $\lambda / \mu$.

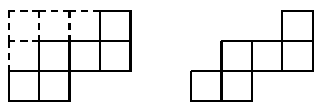
\begin{center}
\setlength{\unitlength}{0.4mm}
\begin{picture}(100,40)(-30,-5)

   \put(-30,0){\line(1,0){20}}
   \put(-30,10){\line(1,0){40}}
   \put(-20,20){\line(1,0){30}}
   \put(0,30){\line(1,0){10}}

   \put(-30,0){\line(0,1){10}}
   \put(-20,0){\line(0,1){20}}
   \put(-10,0){\line(0,1){20}}
   \put(0,10){\line(0,1){20}}
   \put(10,30){\line(0,-1){20}}

\put(-30,10){\dashbox{2}(10,10){}}
\put(-30,20){\dashbox{2}(10,10){}}
\put(-30,20){\dashbox{2}(10,10){}}
\put(-20,20){\dashbox{2}(10,10){}}
\put(-10,20){\dashbox{2}(10,10){}}

   \put(30,0){\line(1,0){20}}
   \put(30,10){\line(1,0){40}}
   \put(40,20){\line(1,0){30}}
   \put(60,30){\line(1,0){10}}

   \put(30,0){\line(0,1){10}}
   \put(40,0){\line(0,1){20}}
   \put(50,0){\line(0,1){20}}
   \put(60,10){\line(0,1){20}}
   \put(70,30){\line(0,-1){20}}

\end{picture}
\end{center}

\end{example}

The number of boxes that appears in a given row or a given column of a diagram is called the \textit{length} of that row or column. 
The \textit{length} of a diagram $D$ is the number of rows of $D$ and the \textit{width} of a diagram $D$ is the number of columns of $D$. 
Given any diagram $D$, the $180^{\circ}$ rotation of a diagram diagram $D$ is denoted by $D^{\circ}$.

For two boxes $b_1$ and $b_2$ in a diagram $D$, we define a \textit{path} from $b_1$ to $b_2$ in $D$ to be a sequence of steps either up, down, left, or right that begins at $b_1$, ends at $b_2$, and at no time leaves the diagram $D$.
We say that a diagram $D$ is \textit{connected} if for any two boxes $b_1$ and $b_2$ of $D$ there is a path from $b_1$ to $b_2$ in $D$. If $D$ is not connected we say it is \textit{disconnected}.

Given two diagrams $D_1$ and $D_2$, if the last column of $D_1$ and first column of $D_2$ are both of length $\geq i$, then the \textit{near-concatenation of depth $i$ of $D_1$ and $D_2$}, denoted $D_1 \odot_i D_2$, is the is the skew diagram obtained by placing $D_1$ and $D_2$ so that the top-right box of $D_1$ is one step left and $i-1$ steps up from the bottom-left box of $D_2$. 
 
\begin{example} Let $D_1$ be the skew diagram $(2,2,2) / (1,1)$ and $D_2$ be the Ferrers diagram $(4,4,2)$. 

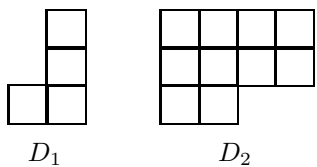
\setlength{\unitlength}{0.5mm}

\begin{picture}(120,50)(-15,-5)

\put(65,0){$D_1$}
\put(115,0){$D_2$}

\put(100,30){\framebox(10,10)[tl]{ }}
\put(110,30){\framebox(10,10)[tl]{ }}
\put(120,30){\framebox(10,10)[tl]{ }}
\put(130,30){\framebox(10,10)[tl]{ }}

\put(100,20){\framebox(10,10)[tl]{ }}
\put(110,20){\framebox(10,10)[tl]{ }}
\put(120,20){\framebox(10,10)[tl]{ }}
\put(130,20){\framebox(10,10)[tl]{ }}

\put(100,10){\framebox(10,10)[tl]{ }}
\put(110,10){\framebox(10,10)[tl]{ }}

\put(70,30){\framebox(10,10)[tl]{ }}
\put(70,20){\framebox(10,10)[tl]{ }}
\put(70,10){\framebox(10,10)[tl]{ }}
\put(60,10){\framebox(10,10)[tl]{ }}

\end{picture}

Here we show the diagrams $D_1 \odot_1 D_2$, $D_1 \odot_2 D_2$, and $D_1 \odot_3 D_2$.

\

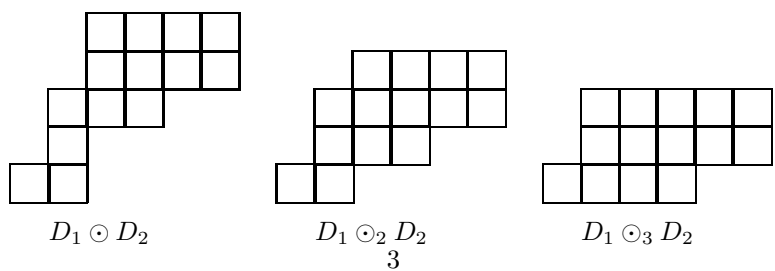
\setlength{\unitlength}{0.5mm}

\begin{picture}(120,50)(30,-5)

\put(50,-30){$D_1 \odot D_2$}

\put(120,-30){$D_1 \odot_2 D_2$}

\put(190,-30){$D_1 \odot_3 D_2$}

\put(60,20){\framebox(10,10)[tl]{ }}
\put(70,20){\framebox(10,10)[tl]{ }}
\put(80,20){\framebox(10,10)[tl]{ }}
\put(90,20){\framebox(10,10)[tl]{ }}

\put(60,10){\framebox(10,10)[tl]{ }}
\put(70,10){\framebox(10,10)[tl]{ }}
\put(80,10){\framebox(10,10)[tl]{ }}
\put(90,10){\framebox(10,10)[tl]{ }}

\put(60,0){\framebox(10,10)[tl]{ }}
\put(70,0){\framebox(10,10)[tl]{ }}

\put(50,0){\framebox(10,10)[tl]{ }}
\put(50,-10){\framebox(10,10)[tl]{ }}
\put(50,-20){\framebox(10,10)[tl]{ }}
\put(40,-20){\framebox(10,10)[tl]{ }}

\put(130,10){\framebox(10,10)[tl]{ }}
\put(140,10){\framebox(10,10)[tl]{ }}
\put(150,10){\framebox(10,10)[tl]{ }}
\put(160,10){\framebox(10,10)[tl]{ }}

\put(130,0){\framebox(10,10)[tl]{ }}
\put(140,0){\framebox(10,10)[tl]{ }}
\put(150,0){\framebox(10,10)[tl]{ }}
\put(160,0){\framebox(10,10)[tl]{ }}

\put(130,-10){\framebox(10,10)[tl]{ }}
\put(140,-10){\framebox(10,10)[tl]{ }}

\put(120,0){\framebox(10,10)[tl]{ }}
\put(120,-10){\framebox(10,10)[tl]{ }}
\put(120,-20){\framebox(10,10)[tl]{ }}
\put(110,-20){\framebox(10,10)[tl]{ }}

\put(200,0){\framebox(10,10)[tl]{ }}
\put(210,0){\framebox(10,10)[tl]{ }}
\put(220,0){\framebox(10,10)[tl]{ }}
\put(230,0){\framebox(10,10)[tl]{ }}

\put(200,-10){\framebox(10,10)[tl]{ }}
\put(210,-10){\framebox(10,10)[tl]{ }}
\put(220,-10){\framebox(10,10)[tl]{ }}
\put(230,-10){\framebox(10,10)[tl]{ }}

\put(200,-20){\framebox(10,10)[tl]{ }}
\put(210,-20){\framebox(10,10)[tl]{ }}

\put(190,0){\framebox(10,10)[tl]{ }}
\put(190,-10){\framebox(10,10)[tl]{ }}
\put(190,-20){\framebox(10,10)[tl]{ }}
\put(180,-20){\framebox(10,10)[tl]{ }}

\end{picture}

\end{example}

\

\

\

Given diagrams $D_1$ and $D_2$, we define their \textit{direct sum} to be the skew diagram $D= D_1 \oplus D_2$ that consists of the subdiagrams $D_1$ and $D_2$ such that the top-right box of $D_1$ is one step left and one step down from the bottom-left box of $D_{2}$. 
We note that we may write $\oplus = \odot_0$.

\begin{example} With $D_1 = (2,2,2) / (1,1)$ and $D_2 = (4,4,2)$, as in the previous example, we now show the diagram $D_1 \oplus D_2$. 

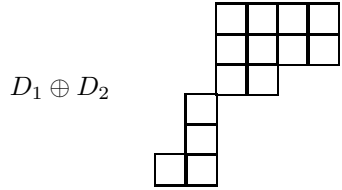
\setlength{\unitlength}{0.4mm}

\begin{picture}(120,50)(-35,-5)

\put(32,10){$D_1 \oplus D_2$}

\put(100,30){\framebox(10,10)[tl]{ }}
\put(110,30){\framebox(10,10)[tl]{ }}
\put(120,30){\framebox(10,10)[tl]{ }}
\put(130,30){\framebox(10,10)[tl]{ }}

\put(100,20){\framebox(10,10)[tl]{ }}
\put(110,20){\framebox(10,10)[tl]{ }}
\put(120,20){\framebox(10,10)[tl]{ }}
\put(130,20){\framebox(10,10)[tl]{ }}

\put(100,10){\framebox(10,10)[tl]{ }}
\put(110,10){\framebox(10,10)[tl]{ }}

\put(90,0){\framebox(10,10)[tl]{ }}
\put(90,-10){\framebox(10,10)[tl]{ }}
\put(90,-20){\framebox(10,10)[tl]{ }}
\put(80,-20){\framebox(10,10)[tl]{ }}

\end{picture}

\end{example}

\

If $D$ is a diagram, then a \textit{tableau}---plural \textit{tableaux}---$\mathcal{T}$ \textit{of shape $D$} is the array obtained by filling the boxes of the $D$ with the positive integers, where repetition is allowed. 
A tableau is said to be a \textit{semistandard Young tableau} (or simply \textit{semistandard}, for short) if each row gives a weakly increasing sequence of integers and each columns gives a strictly increasing sequence of integers.
We will often abbreviate ``semistandard Young tableau" by SSYT and ``semistandard Young tableaux" by SSYTx. 

When we wish to depict a certain tableau we will either show the underlying diagram with the entries residing in the boxes of the diagram or we may simply replace the boxes with the entries, so that the entries themselves depict the underlying shape of the tableau.

The \textit{content} of a tableau $\mathcal{T}$ is the weak composition given by
\[ \nu(\mathcal{T}) = ( \# \textrm{1's in }\mathcal{T}, \# \textrm{2's in }\mathcal{T}, \ldots).\]

Now, given a tableau $\mathcal{T}$ of shape $D$, we may wish to focus on the entries of $\mathcal{T}$ that lie in some subdiagram $D'$ of $D$. 
In this way we obtain a \textit{subtableau} of shape $D'$.
Similarly, given a tableau $\mathcal{T}'$ of shape $D'$ and a diagram $D$ containing $D'$, we may sometimes wish to fill the copy of $D'$ within $D$ as in the tableau $\mathcal{T}'$. 

\

Given a skew diagram $D$, the \textit{skew Schur function corresponding to $D$} is defined to be \label{slambda}
\begin{equation}
\label{schurdef}
 s_{D}(\textbf{x}) = \sum_{\mathcal{T}} {x_1}^{\# \textrm{1's in } \mathcal{T}}{x_2}^{\# \textrm{2's in } \mathcal{T}} \cdots,
\end{equation}
where the sum is taken over all semistandard Young tableaux $\mathcal{T}$ of shape $D$. 
When $D=\lambda$ is a partition, $s_\lambda$ is called the \textit{Schur function corresponding to $\lambda$}.

The set $\{s_\lambda | \lambda \vdash n \}$ is a basis of $\Lambda^n$, the set of homogeneous symmetric functions of degree $n$. 
Thus any homogeneous symmetric function of degree $n$ can be written as a unique linear combination of the $s_\lambda$. 
Therefore for each $f \in \Lambda^n$ we can write $f=\sum_{\lambda} a_\lambda s_\lambda$ for appropriate coefficients. 
We use the operator $[s_\lambda]$ to extract the coefficient of $s_\lambda$. That is, if $f=\sum_{\lambda} a_\lambda s_\lambda$ then $[s_\lambda] (f) = a_\lambda$.

As was mentioned in the introduction, for any partitions $\mu$ and $\nu$ we have
\begin{equation}
\label{smusnu}
 s_\mu s_\nu = \sum_{\lambda \vdash n} c_{\mu \nu}^{\lambda} s_\lambda, 
\end{equation}
and for any skew diagram $\lambda / \mu$ we have
\begin{equation}
\label{slambdaskewmu}
 s_{\lambda / \mu} = \sum_{\nu \vdash n} c_{\mu \nu}^{\lambda} s_\nu 
\end{equation}
where the $c_{\mu \nu}^{\lambda}$ are the \textit{Littlewood-Richardson coefficients}.
It turns out that the Littlewood-Richardson coefficients are non-negative integers and they count an interesting class of SSYT that we shall now describe.

Given a tableau $\mathcal{T}$, the \textit{reading word} of $\mathcal{T}$ is the sequence of integers obtained by reading the entries of the rows of $\mathcal{T}$ from right to left, proceeding from the top row to the bottom. 
We say that a sequence $r=r_1,r_2,\ldots, r_k$ is \textit{lattice} if, for each $j$, when reading the sequence from left to right the number of $j$'s that we have read is never less than the number of $j+1$'s that we have read.

\begin{theorem} [Littlewood-Richardson Rule] (\cite{LR})
\label{lr}

For partitions $\lambda, \mu$, and $\nu$, the \textit{Littlewood-Richardson coefficient} $c_{\mu \nu}^{\lambda}$ is the number of SSYTx of shape $\lambda / \mu$, content $\nu$, with lattice reading word.
\end{theorem}

For any $f = \sum_{\lambda \vdash n} a_{\lambda} s_{\lambda} \in \Lambda^n$, we say that $f$ is \textit{Schur-positive}, and write $f \geq_s 0$, if each $a_{\lambda} \geq 0$. 
Therefore, the Littlewood-Richardson rule shows that both $s_{\mu} s_{\nu}$ and $s_{\lambda / \mu}$ are Schur-positive. 
For $f,g \in \Lambda^n$, we will be interested in whether or not the difference $f-g$ is Schur positive. 
We shall write $f \geq_s g$ whenever $f-g$ is Schur-positive. 

If we consider the relation $\succeq_s$ on the set of all Schur-equivalent classes of diagrams (i.e. $[D]_s= \{D' | s_{D}=s_{D'}\}$), then $\succeq_s$ defines a partial ordering.
This allows us to view the Hasse diagram for the relation $\succeq_s$ on the set of these Schur-equivalent classes.
Some work in determining these equivalence classes includes \cite{{btvw},{g3},{mvw1},{rsvw}}.

\

We close these preliminaries by mentioning several useful results regarding skew Schur functions.
\begin{theorem}{(\cite{stanley}, Exercise 7.56(a))}
\label{rotate}
Given a skew diagram $D$, 
\begin{equation} s_D= s_{D^{\circ}}. 
\end{equation}
\end{theorem}

\begin{theorem}
\label{disjprod}
The Schur function of any disconnected skew diagram is reducible. If $D = D_1 \oplus D_2$, then we have 
\begin{equation}
s_{D}=s_{D_1} s_{D_2}.
\end{equation}
\end{theorem}
\begin{proof} Any semistandard Young tableau of shape $D_1 \oplus D_2$ gives rise to semistandard Young tableaux of shape $D_1$ and $D_2$ by taking the obvious subtableaux. 
Conversely, any semistandard Young tableaux $\mathcal{T}_1$ of shape $D_1$ and $\mathcal{T}_2$ of shape $D_2$ give rise to the tableau $\mathcal{T}_1 \oplus \mathcal{T}_2$ of shape $D_1 \oplus D_2$, which is clearly semistandard. \qed
\end{proof}

Given a partition $\rho$ contained in the $a \times b$ rectangle $(b^a)$ we define the \textit{complementary partition $\rho^c$ in the rectangle $(b^a)$} by  $\rho^c = ((b^a) / \rho)^\circ$. 
That is, $\rho^c$ is the complement of $\rho$ in $(b^a)$ rotated by $180^\circ$. 
It is easy to see that this definition does define a partition. 
We display the relevant diagrams below for clarity.

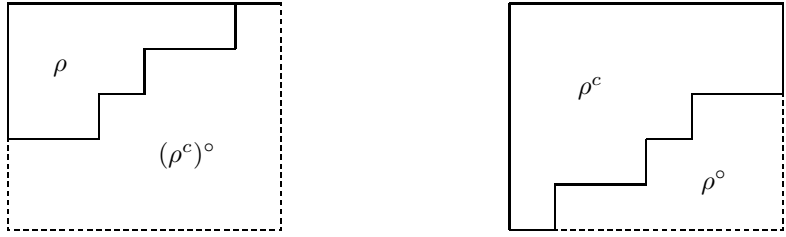
\setlength{\unitlength}{0.6mm}

\begin{picture}(100,60)(150,-25)

\put(170,5){$\rho$}
\put(193,-15){$(\rho^c)^\circ $}

\put(160,20){\line(1,0){60}}

\put(160,-10){\line(1,0){20}}
\put(180,0){\line(1,0){10}}
\put(190,10){\line(1,0){20}}

\put(160,-10){\line(0,1){30}}
\put(180,-10){\line(0,1){10}}
\put(190,0){\line(0,1){10}}
\put(210,10){\line(0,1){10}}

\put(160,-30){\dashbox{1}(60,50)[tl]{ }}

\put(285,0){$\rho^c$}
\put(312,-21){$\rho^\circ $}

\put(270,20){\line(1,0){60}}

\put(270,-30){\line(1,0){10}}
\put(280,-20){\line(1,0){20}}
\put(300,-10){\line(1,0){10}}
\put(310,0){\line(1,0){20}}

\put(270,-30){\line(0,1){50}}
\put(280,-30){\line(0,1){10}}
\put(300,-20){\line(0,1){10}}
\put(310,-10){\line(0,1){10}}
\put(330,0){\line(0,1){20}}

\put(270,-30){\dashbox{1}(60,50)[tl]{ }}

\end{picture}

\

\

For what follows, we shall make use of the following fact.

\begin{theorem} (\cite{complementcite})
\label{firstcut}
Let $\rho$ be a partition contained in the $a \times b$ rectangle $(b^a)$, $\kappa \subset \rho$ be a second partition. Then the skew diagram $\rho / \kappa$ satisfies 
\[ s_{\rho / \kappa} = \sum_{\nu \subseteq (b^a)} c_{\kappa \rho^c}^{\nu} s_{\nu^c}, \] where $c_{\kappa \rho^c}^{\nu}$ are the Littlewood-Richardson coefficients.
\end{theorem}

Given a symmetric function $f=\sum_{\nu } a_{\nu} s_{\nu}$ we define the \textit{truncated complement of $f$ in the rectangle $(b^a)$} as 
\begin{equation}
\label{cf}
c(f) = \sum_{\nu \subseteq (b^a)} a_{\nu} s_{\nu^c}.
\end{equation}
The rectangle being used should be clear from the context if it is not specifically mentioned.

We note that one effect of the restriction $\nu \subseteq (b^a)$ is that in passing from $f$ to $c(f)$, we are reducing computations to $a$ variables. 
This follows since any term $s_{\nu}$ from $f$ with $l(\nu)>a$ is effectively being set to zero, which is precisely what happens when reducing $f$ to $a$ variables.

\

We may now restate Theorem~\ref{firstcut} as follows.

\begin{corollary}
\label{rectcor}
Let $\rho$ be a partition contained in the $a \times b$ rectangle $(b^a)$, and $\kappa \subset \rho$ be a second partition. Then the skew diagram $\rho / \kappa$ satisfies 
\[ s_{\rho / \kappa} = c(s_\kappa s_{\rho^c}). \]
\end{corollary}

\begin{proof}
From the definition of the Littlewood-Richardson numbers, we have
\[ s_{\kappa} s_{\rho^c} =  \sum_{\nu} c_{\kappa \rho^c}^{\nu} s_{\nu}.\]
Hence \[c(s_{\kappa} s_{\rho^c}) = c(\sum_{\nu} c_{\kappa \rho^c}^{\nu} s_{\nu}) = \sum_{\nu \subseteq (b^a)} c_{\kappa \rho^c}^{\nu} s_{\nu^c}. \]
By Theorem~\ref{firstcut}, this is just $s_{\rho / \kappa}$, so we are done.\qed
\end{proof}

One must be careful in working with the expression $c(s_\kappa s_{\rho^c})$. 
We must truncate any term $s_{\nu}$ we obtain in from the product with $\nu \not\subseteq (b^a)$. 
Thus we must be mindful of the dimensions of the rectangle we are working in.

\section{Staircases and Fat Staircases with Bad Foundations}

A Ferrers diagram is a \textit{staircase} if it is the Ferrers diagram of a partition of the form $\lambda = (n,n-1,n-2,\ldots, 2,1)$ or if it is the $180^{\circ}$ rotation of such a diagram. 
Both these diagrams are referred to as \textit{staircases of length $n$} and will be denoted by $\delta_n$ and $\Delta_n$ respectively.

\begin{example} Here we see the two staircases of length 5.

\

\

\

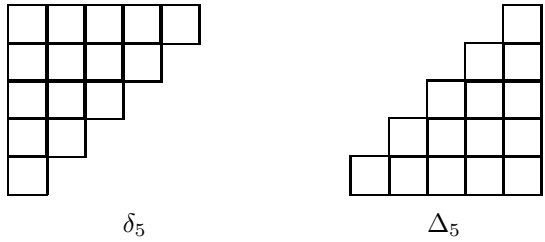
\setlength{\unitlength}{0.5mm}

\begin{picture}(100,50)(-30,-5)

\put(-10,50){\framebox(10,10)[tl]{ }}
\put(0,50){\framebox(10,10)[tl]{ }}
\put(10,50){\framebox(10,10)[tl]{ }}
\put(20,50){\framebox(10,10)[tl]{ }}
\put(30,50){\framebox(10,10)[tl]{ }}
     
\put(-10,40){\framebox(10,10)[tl]{ }}
\put(0,40){\framebox(10,10)[tl]{ }}
\put(10,40){\framebox(10,10)[tl]{ }}
\put(20,40){\framebox(10,10)[tl]{ }}
 
\put(-10,30){\framebox(10,10)[tl]{ }}
\put(0,30){\framebox(10,10)[tl]{ }}
\put(10,30){\framebox(10,10)[tl]{ }}

\put(-10,20){\framebox(10,10)[tl]{ }}
\put(0,20){\framebox(10,10)[tl]{ }}

\put(-10,10){\framebox(10,10)[tl]{ }}

\put(20,0){$\delta_5$}

\put(100,0){$\Delta_5$}

\put(120,50){\framebox(10,10)[tl]{ }}
\put(120,40){\framebox(10,10)[tl]{ }}
\put(120,30){\framebox(10,10)[tl]{ }}
\put(120,20){\framebox(10,10)[tl]{ }}
\put(120,10){\framebox(10,10)[tl]{ }}

\put(110,40){\framebox(10,10)[tl]{ }}
\put(110,30){\framebox(10,10)[tl]{ }}
\put(110,20){\framebox(10,10)[tl]{ }}
\put(110,10){\framebox(10,10)[tl]{ }}

\put(100,30){\framebox(10,10)[tl]{ }}
\put(100,20){\framebox(10,10)[tl]{ }}
\put(100,10){\framebox(10,10)[tl]{ }}

\put(90,20){\framebox(10,10)[tl]{ }}
\put(90,10){\framebox(10,10)[tl]{ }}

\put(80,10){\framebox(10,10)[tl]{ }}

\end{picture}

\end{example}

Given a composition $\alpha=(\alpha_1, \ldots,\alpha_n)$, we let 
\[ \delta_\alpha= (n^{\alpha_n},{n-1}^{\alpha_{n-1}}, \ldots,2^{\alpha_2}, 1^{\alpha_1}) \textrm{ and } \Delta_\alpha= (n^{\alpha_n},{n-1}^{\alpha_{n-1}}, \ldots,2^{\alpha_2}, 1^{\alpha_1})^\circ. \]
We call a skew diagram $D$ a \textit{fat staircase} if $D=\delta_{\alpha}$ or $D=\Delta_{\alpha}$ for some composition $\alpha$.
The numbers $\alpha_i$ count the number of rows of $D$ with $i$ boxes, for each $i$.
Using this notation the regular staircases may be expressed as $\delta_n = \delta_{(1^n)}$ and $\Delta_n = \Delta_{(1^n)}$, respectively. 
Both fat staircases $\delta_\alpha$ and $\Delta_\alpha$ have width $= l(\alpha)$ and length $= |\alpha| = \sum_{i=1}^{n} \alpha_i$. 

\begin{example}
Here we see the the fat staircases $\delta_{(1,2,2)}$ and $\Delta_{(3,1,2,3)}$.

\

\

\

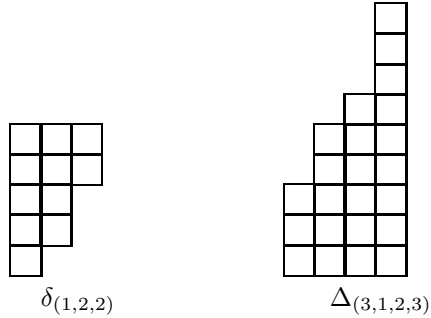
\setlength{\unitlength}{0.4mm}

\begin{picture}(100,50)(-50,-15)

\put(10,10){\framebox(10,10)[tl]{ }}
\put(20,10){\framebox(10,10)[tl]{ }}
\put(30,10){\framebox(10,10)[tl]{ }}
     
\put(10,0){\framebox(10,10)[tl]{ }}
\put(20,0){\framebox(10,10)[tl]{ }}
\put(30,0){\framebox(10,10)[tl]{ }}
 
\put(10,-10){\framebox(10,10)[tl]{ }}
\put(20,-10){\framebox(10,10)[tl]{ }}

\put(10,-20){\framebox(10,10)[tl]{ }}
\put(20,-20){\framebox(10,10)[tl]{ }}

\put(10,-30){\framebox(10,10)[tl]{ }}

\put(20,-40){$\delta_{(1,2,2)}$}

\put(115,-40){$\Delta_{(3,1,2,3)}$}

\put(100,-30){\framebox(10,10)[tl]{ }}
\put(110,-30){\framebox(10,10)[tl]{ }}
\put(120,-30){\framebox(10,10)[tl]{ }}
\put(130,-30){\framebox(10,10)[tl]{ }}

\put(100,-20){\framebox(10,10)[tl]{ }}
\put(110,-20){\framebox(10,10)[tl]{ }}
\put(120,-20){\framebox(10,10)[tl]{ }}
\put(130,-20){\framebox(10,10)[tl]{ }}

\put(100,-10){\framebox(10,10)[tl]{ }}
\put(110,-10){\framebox(10,10)[tl]{ }}
\put(120,-10){\framebox(10,10)[tl]{ }}
\put(130,-10){\framebox(10,10)[tl]{ }}

\put(110,0){\framebox(10,10)[tl]{ }}
\put(120,0){\framebox(10,10)[tl]{ }}
\put(130,0){\framebox(10,10)[tl]{ }}

\put(110,10){\framebox(10,10)[tl]{ }}
\put(120,10){\framebox(10,10)[tl]{ }}
\put(130,10){\framebox(10,10)[tl]{ }}

\put(120,20){\framebox(10,10)[tl]{ }}
\put(130,20){\framebox(10,10)[tl]{ }}

\put(130,30){\framebox(10,10)[tl]{ }}

\put(130,40){\framebox(10,10)[tl]{ }}

\put(130,50){\framebox(10,10)[tl]{ }}

\end{picture}

\end{example}

\

\

\

Given a composition $\alpha$, $k \geq 0$, and partitions $\lambda$, $\mu$ with $\lambda_1 - \mu_1 -k \leq l(\alpha)$ we now define $\mathcal{S}(\lambda,\mu, \alpha^\triangleleft;k)$ to be the diagram obtained by placing $\lambda / \mu$ immediately below $\Delta_{\alpha}$ such that the rows of the two diagrams overlap in precisely $\lambda_1 - \mu_1 -k$ positions.
We call $\mathcal{S}(\lambda, \mu, \alpha^\triangleleft;k)$ a \textit{fat staircase with bad foundation}.
The subdiagram $\lambda / \mu$ is called the foundation of $\mathcal{S}(\lambda,\mu,\alpha^\triangleleft;k)$.

The fact that $\Delta_{\alpha}$ and $\lambda / \mu$ overlap in precisely $\lambda_1 - \mu_1 -k$ positions means that the first row of $\lambda / \mu$ begins exactly one box below and $k$ boxes left of the bottom-left box of the diagram $\Delta_{\alpha}$.
When we wish to have a partition as the foundation we shall write $\mu = \emptyset$ and simply use $\mathcal{S}(\lambda,\alpha^\triangleleft;k)$ in place of $\mathcal{S}(\lambda,\mu,\alpha^\triangleleft;k)$.

\begin{example} If we take $\alpha = (1,1,3,1,2,1)$, $\lambda = (6,5,5,5,3)$, $\mu = \emptyset$, and $k=0$, then we obtain the following staircase with bad foundation $\mathcal{S}(\lambda, \alpha^\triangleleft;k)$.

\

\

\

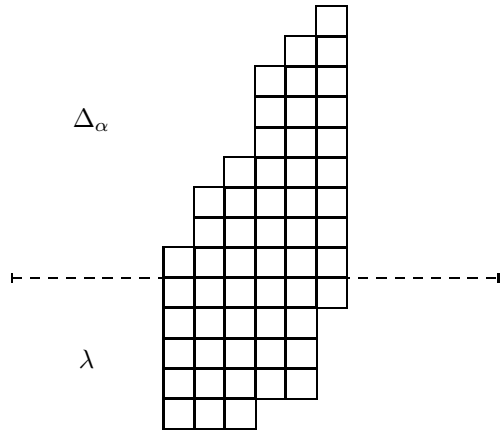
\setlength{\unitlength}{0.4mm}

\begin{picture}(000,140)(110,-70)

\put(180,30){$\Delta_{\alpha}$}

\put(182,-50){$\lambda$}

\put(160,-20){\dashbox{3}(160,0)[tl]{ }}

\put(210,-70){\framebox(10,10)[tl]{ }}
\put(220,-70){\framebox(10,10)[tl]{ }}
\put(230,-70){\framebox(10,10)[tl]{ }}

\put(210,-60){\framebox(10,10)[tl]{ }}
\put(220,-60){\framebox(10,10)[tl]{ }}
\put(230,-60){\framebox(10,10)[tl]{ }}
\put(240,-60){\framebox(10,10)[tl]{ }}
\put(250,-60){\framebox(10,10)[tl]{ }}

\put(210,-50){\framebox(10,10)[tl]{ }}
\put(220,-50){\framebox(10,10)[tl]{ }}
\put(230,-50){\framebox(10,10)[tl]{ }}
\put(240,-50){\framebox(10,10)[tl]{ }}
\put(250,-50){\framebox(10,10)[tl]{ }}

\put(210,-40){\framebox(10,10)[tl]{ }}
\put(220,-40){\framebox(10,10)[tl]{ }}
\put(230,-40){\framebox(10,10)[tl]{ }}
\put(240,-40){\framebox(10,10)[tl]{ }}
\put(250,-40){\framebox(10,10)[tl]{ }}

\put(210,-30){\framebox(10,10)[tl]{ }}
\put(220,-30){\framebox(10,10)[tl]{ }}
\put(230,-30){\framebox(10,10)[tl]{ }}
\put(240,-30){\framebox(10,10)[tl]{ }}
\put(250,-30){\framebox(10,10)[tl]{ }}
\put(260,-30){\framebox(10,10)[tl]{ }}

\put(210,-20){\framebox(10,10)[tl]{ }}
\put(220,-20){\framebox(10,10)[tl]{ }}
\put(230,-20){\framebox(10,10)[tl]{ }}
\put(240,-20){\framebox(10,10)[tl]{ }}
\put(250,-20){\framebox(10,10)[tl]{ }}
\put(260,-20){\framebox(10,10)[tl]{ }}

\put(220,-10){\framebox(10,10)[tl]{ }}
\put(230,-10){\framebox(10,10)[tl]{ }}
\put(240,-10){\framebox(10,10)[tl]{ }}
\put(250,-10){\framebox(10,10)[tl]{ }}
\put(260,-10){\framebox(10,10)[tl]{ }}

\put(220,0){\framebox(10,10)[tl]{ }}
\put(230,0){\framebox(10,10)[tl]{ }}
\put(240,0){\framebox(10,10)[tl]{ }}
\put(250,0){\framebox(10,10)[tl]{ }}
\put(260,0){\framebox(10,10)[tl]{ }}

\put(230,10){\framebox(10,10)[tl]{ }}
\put(240,10){\framebox(10,10)[tl]{ }}
\put(250,10){\framebox(10,10)[tl]{ }}
\put(260,10){\framebox(10,10)[tl]{ }}

\put(240,20){\framebox(10,10)[tl]{ }}
\put(250,20){\framebox(10,10)[tl]{ }}
\put(260,20){\framebox(10,10)[tl]{ }}

\put(240,30){\framebox(10,10)[tl]{ }}
\put(250,30){\framebox(10,10)[tl]{ }}
\put(260,30){\framebox(10,10)[tl]{ }}

\put(240,40){\framebox(10,10)[tl]{ }}
\put(250,40){\framebox(10,10)[tl]{ }}
\put(260,40){\framebox(10,10)[tl]{ }}

\put(250,50){\framebox(10,10)[tl]{ }}
\put(260,50){\framebox(10,10)[tl]{ }}

\put(260,60){\framebox(10,10)[tl]{ }}

\end{picture}

\end{example}

The reason we write $\mathcal{S}(\lambda, \mu, \balpha;k)$ instead of writing $\mathcal{S}(\lambda, \mu,\alpha ;k)$ is to avoid confusion in the cases when the composition $\alpha=(\alpha_1,\alpha_2, \ldots, \alpha_n)$ is weakly decreasing and could be misinterpreted as representing the diagram of a partitions. 
In this case, confusion could arise, since we shall now define $\mathcal{S}(\lambda, \mu, D;k)$ for any skew diagram $D$.

Given a skew diagram  $D$ and partitions $\lambda$, $\mu$ with $\lambda_1 - \mu_1 -k \leq $ the length of the last row of $D$, we now define $\mathcal{S}(\lambda,\mu, D;k)$ to be the diagram obtained by placing $\lambda / \mu$ immediately below $D$ such that the rows of the two diagrams overlap in precisely $\lambda_1 - \mu_1 -k$ positions.
As before, $k$ is the distance to the left that the first row of $\lambda / \kappa$ extends from the last row of $D$.

For example, consider the diagrams $\mathcal{S}((2,1), (2,2);1)$ and $\mathcal{S}((2,1),(2,2)^\triangleleft ;1)$ shown below. 
The first uses the diagram $D=(2,2)$, whereas the second uses the diagram $D=\Delta_{(2,2)}$. 

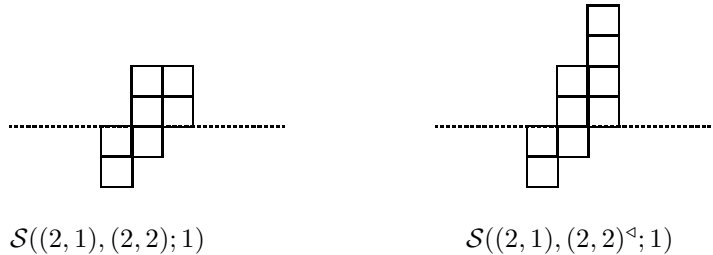
\setlength{\unitlength}{0.4mm}

\begin{picture}(100,80)(-10,-5)

\put(60,40){\framebox(10,10)[tl]{ }}
\put(70,40){\framebox(10,10)[tl]{ }}

\put(60,30){\framebox(10,10)[tl]{ }}
\put(70,30){\framebox(10,10)[tl]{ }}

\put(50,20){\framebox(10,10)[tl]{ }}
\put(60,20){\framebox(10,10)[tl]{ }}

\put(50,10){\framebox(10,10)[tl]{ }}

\put(20,-10){$\mathcal{S}((2,1), (2,2);1)$}

\put(210,60){\framebox(10,10)[tl]{ }}

\put(210,50){\framebox(10,10)[tl]{ }}

\put(200,40){\framebox(10,10)[tl]{ }}
\put(210,40){\framebox(10,10)[tl]{ }}

\put(200,30){\framebox(10,10)[tl]{ }}
\put(210,30){\framebox(10,10)[tl]{ }}

\put(190,20){\framebox(10,10)[tl]{ }}
\put(200,20){\framebox(10,10)[tl]{ }}

\put(190,10){\framebox(10,10)[tl]{ }}

\put(170,-10){$\mathcal{S}((2,1),(2,2)^\triangleleft;1)$}

\put(20,30){\dashbox{1}(90,0)[tl]{ }}

\put(160,30){\dashbox{1}(90,0)[tl]{ }}

\end{picture}

\

\

One of the advantages in computing the skew Schur functions of fat staircases with bad foundations is that, when using the Littlewood-Richarson rule, the $\Delta_\alpha$ portion of the diagram can be filled in only one way. 
By using Theorem~\ref{rotate} we can be see this algebraically from the equation 
\[ s_{\Delta_\alpha} = s_{\Delta_\alpha^\circ} = s_{\delta_\alpha}, \]
 where $s_{\delta_\alpha}$ is a Schur function. 
The unique filling of $\Delta_\alpha$ obeying the semistandard conditions and the lattice condition is easily seen to be the filling that places the entries $1,2,\ldots, l$ into each column of length $l$.

\setlength{\unitlength}{0.5mm}
\begin{picture}(0,0)(0,0)
\put(90,-74){\line(0,1){26} }
\end{picture}
\begin{lemma} 

\label{kfatfirstrowlemma}
Let $\mathcal{S}(\lambda,\mu,\alpha^\triangleleft;k)$ be a fat staircase with bad foundation for some $k \geq 0$ and $\mathcal{T}$ be a SSYT of shape $\mathcal{S}(\lambda,\mu,\alpha^\triangleleft;k)$ whose reading word is lattice.
If $\alpha=(\alpha_1,\dots, \alpha_n)$, then the entries in the first row of the foundation of $\mathcal{T}$ consist values taken from the set 
\[R_{\alpha,k} = \left\{ 1+\sum_{i=1}^{j} \alpha_{n+1-i} \textrm{ } \textrm{ } \textrm{ } j= 1, 2, \ldots, n  \right\} \cup \left\{ \begin{array}{cll}
\{1\} & \mbox{if} & k > 0 \\
\emptyset &  \mbox{if} & k=0. 
\end{array}\right.\]
Furthermore, the value $1$ can occur at most $k$ times and the rest of the values can appear at most once.
\end{lemma}

\begin{proof}

Let $R$ be the first row of the foundation of $\mathcal{T}$ and $t \in R$.

Since $\mathcal{T}$ is a SSYT, the columns strictly increase. Thus $t = 1$ is allowed if and only if $k \geq 1$ since it is precisely in that case that the first value in $R$ is not below an entry of $\Delta_{\alpha}$. 
Furthermore, since there are only $k$ boxes from the first row of the foundation of $\mathcal{T}$ that extend out from $\Delta_\alpha$, there can be at most $k$ $1$'s in $R$. 

If $t > 1$ then, when reading the row $R$ from right to left, the lattice condition implies that there is at least one more $t-1$ in $\Delta_{\alpha}$ than there are $t$'s in $\Delta_{\alpha}$. 
Since the content of $\Delta_{\alpha}$ is $(n^{\alpha_n},{n-1}^{\alpha_{n-1}}, \ldots, 1^{\alpha_1})$, the only instances when this occurs are when $t=1+\sum_{i=1 \ldots j} \alpha_{n+1-j}$ for $j=1,2,\ldots,n$.
Therefore every entry of $R$ is an element of $R_{\alpha,k}$.
Further, if a value $t>1$ appeared twice in $R$, then the lattice condition would be violated. 
Hence each $t \in R_{\alpha,k}$, $t \neq 1$, can appear at most once in $R$. \qed

\end{proof}

The next result tells us when we may obtain a SSYT of shape $\mathcal{S}(\lambda,\mu,\alpha^\triangleleft;k)$ with lattice reading word from a SSYT of shape $\lambda / \mu  \oplus \Delta_\alpha$ with lattice reading word. 

\begin{lemma} 
\label{kfatjoinlemma}
Let $\alpha$ be a composition, $\lambda$ and $\mu$ be partitions, and $k \geq 0$ such that $\lambda_1 - \mu_1 -k \leq l(\alpha)$.
If $T$ is a SSYT of shape $\lambda / \mu  \oplus \Delta_\alpha$ with lattice reading word such that there are at most $k$ $1$'s in the first row of $\lambda / \mu$, then the tableau of shape $\mathcal{S}(\lambda,\mu,\alpha^\triangleleft;k)$ obtained from $T$ by shifting the foundation $\lambda / \mu$ to the right is also a SSYT with lattice reading word.

\end{lemma}

\begin{proof}
Let $T$ be a SSYT of shape $\lambda / \mu  \oplus \Delta_\alpha$ with lattice reading word and let $T_k'$ be the tableau of shape $\mathcal{S}(\lambda,\mu,\alpha^\triangleleft;k)$ obtained from $T$ by shifting the foundation $\lambda$ to the right.
Since shifting $\lambda / \mu$ to the right does not affect the order in which the entries are read, $T_k'$ has a lattice reading word.
Also, the rows of $T_k'$ weakly increase since they are the same as the rows of $T$.
Further, to check that the columns of $T_k'$ strictly increase, we need only check that they strictly increase at the positions where the two subdiagrams $\Delta_{\alpha}$ and $\lambda / \mu$ are joined.

Let $R$ denote the first row of $\lambda / \mu$ and $\alpha= (\alpha_1, \ldots, \alpha_n)$. 
As in the proof of Lemma~\ref{kfatfirstrowlemma}, the lattice condition on $T$ implies that the entries of $R$ consist of values of $R_{\alpha, k}$. 
Further, the value $1$ can occur at most $k$ times and the rest of the values of $R$ are distinct.
Let $q$ be the number of times $1$ appears in $R$, so that $k \geq q$. 
Further, let $r_1 \leq r_2 \leq \ldots \leq r_{\lambda_1}$ be the entries of $R$. 

\

Consider the case $k \geq 1$. 
Since $r_1=r_2 =\ldots = r_q =1$, we have $r_q = \textrm{min}(R_{\alpha,k})$ and for each $1 \leq j \leq n$ we have  
\[ r_{j+q} \geq \textrm{ the } (j+1) \textrm{-th smallest value of } R_{\alpha,k} = 1 + \sum_{i=1}^{j} \alpha_{n+1-i}. \] 
Since $k \geq q$, for each $1 \leq j \leq n$ we have 
\[ r_{j+k} \geq r_{j+q} \geq 1 + \sum_{i=1}^{j} \alpha_{n+1-i}. \]
As illustrated in the diagram below, the entry $r_{j+k}$ is beneath $\sum_{i=1}^{j} \alpha_{n+1-i}$ boxes. 
From the unique filling of $\Delta_{\alpha}$, the entry of $\Delta_\alpha$ directly above $r_{j+k}$ is $\sum_{i=1}^{j} \alpha_{n+1-i}$.

\

\

\

\

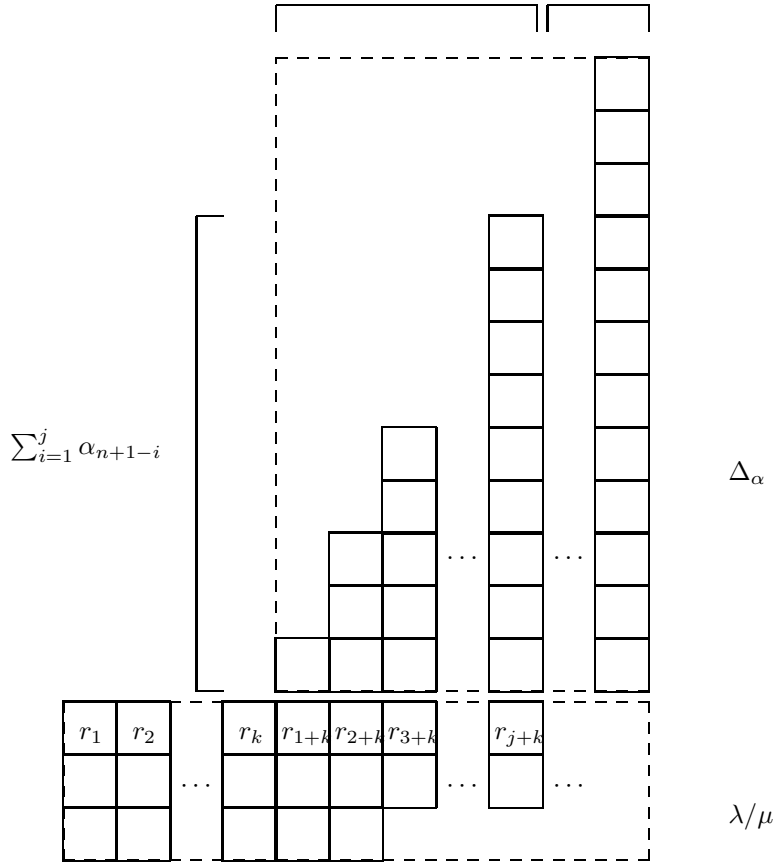
\setlength{\unitlength}{0.7mm}

\begin{picture}(100,80)(30,0)

\put(115,105){$j$}

\put(142,105){$n-j$}

\put(40,15){$\sum_{i=1}^{j} \alpha_{n+1-i}$}

\put(75,-30){\line(0,1){90}}

\put(75,-30){\line(1,0){5}}
\put(75,60){\line(1,0){5}}

\put(90,100){\line(1,0){49}}
\put(141,100){\line(1,0){19}}

\put(90,100){\line(0,-1){5}}
\put(139,100){\line(0,-1){5}}

\put(141,100){\line(0,-1){5}}
\put(160,100){\line(0,-1){5}}

\put(90,-30){\dashbox{2}(70,120)[tl]{ }}

\put(50,-62){\dashbox{2}(110,30)[tl]{ }}

\put(175,10){$\Delta_\alpha$}
\put(175,-55){$\lambda / \mu $}

\put(90,-30){\framebox(10,10)[tl]{ }}
\put(100,-30){\framebox(10,10)[tl]{ }}
\put(110,-30){\framebox(10,10)[tl]{ }}
\put(130,-30){\framebox(10,10)[tl]{ }}
\put(150,-30){\framebox(10,10)[tl]{ }}

\put(100,-20){\framebox(10,10)[tl]{ }}
\put(110,-20){\framebox(10,10)[tl]{ }}
\put(130,-20){\framebox(10,10)[tl]{ }}
\put(150,-20){\framebox(10,10)[tl]{ }}

\put(100,-10){\framebox(10,10)[tl]{ }}
\put(110,-10){\framebox(10,10)[tl]{ }}
\put(130,-10){\framebox(10,10)[tl]{ }}
\put(150,-10){\framebox(10,10)[tl]{ }}

\put(110,0){\framebox(10,10)[tl]{ }}
\put(130,0){\framebox(10,10)[tl]{ }}
\put(150,0){\framebox(10,10)[tl]{ }}

\put(110,10){\framebox(10,10)[tl]{ }}
\put(130,10){\framebox(10,10)[tl]{ }}
\put(150,10){\framebox(10,10)[tl]{ }}

\put(130,20){\framebox(10,10)[tl]{ }}
\put(150,20){\framebox(10,10)[tl]{ }}

\put(130,30){\framebox(10,10)[tl]{ }}
\put(150,30){\framebox(10,10)[tl]{ }}

\put(130,40){\framebox(10,10)[tl]{ }}
\put(150,40){\framebox(10,10)[tl]{ }}

\put(130,50){\framebox(10,10)[tl]{ }}
\put(150,50){\framebox(10,10)[tl]{ }}

\put(150,60){\framebox(10,10)[tl]{ }}

\put(150,70){\framebox(10,10)[tl]{ }}

\put(150,80){\framebox(10,10)[tl]{ }}

\put(91,-39){$r_{1+k}$}
\put(101,-39){$r_{2+k}$}
\put(111,-39){$r_{3+k}$}
\put(131,-39){$r_{j+k}$}

\put(53,-39){$r_1$}
\put(63,-39){$r_2$}
\put(72,-49){$\cdots$}
\put(83,-39){$r_k$}

\put(50,-42){\framebox(10,10)[tl]{ }}
\put(60,-42){\framebox(10,10)[tl]{ }}
\put(80,-42){\framebox(10,10)[tl]{ }}

\put(90,-42){\framebox(10,10)[tl]{ }}
\put(100,-42){\framebox(10,10)[tl]{ }}
\put(110,-42){\framebox(10,10)[tl]{ }}
\put(130,-42){\framebox(10,10)[tl]{ }}
     
\put(50,-52){\framebox(10,10)[tl]{ }}
\put(60,-52){\framebox(10,10)[tl]{ }}
\put(80,-52){\framebox(10,10)[tl]{ }}

\put(90,-52){\framebox(10,10)[tl]{ }}
\put(100,-52){\framebox(10,10)[tl]{ }}
\put(110,-52){\framebox(10,10)[tl]{ }}
\put(130,-52){\framebox(10,10)[tl]{ }}

\put(50,-62){\framebox(10,10)[tl]{ }}
\put(60,-62){\framebox(10,10)[tl]{ }}
\put(80,-62){\framebox(10,10)[tl]{ }}

\put(90,-62){\framebox(10,10)[tl]{ }}
\put(100,-62){\framebox(10,10)[tl]{ }}

\put(122,-49){$\cdots$}
\put(142,-49){$\cdots$}

\put(122,-5){\ldots}
\put(142,-5){\ldots}

\end{picture}

\newpage

Thus the columns strictly increase. 
Therefore $T_1'$ is a SSYT with lattice reading word, as desired.

\

Now consider the case when $k=0$. Then for each $1 \leq j \leq n$ we have 
\[ r_j \geq j \textrm{-th smallest value of } R_{\alpha, k}  \geq 1 + \sum_{i=1}^{j} \alpha_{n+1-i}. \]
Also, the entry $r_j$ is beneath precisely $\sum_{i=1}^{j} \alpha_{n+1-i}$ boxes, so the entry of $\Delta_\alpha$ directly above $r_j$ is $\sum_{i=1}^{j} \alpha_{n+1-i}$. 
Thus the columns strictly increase. 
Therefore $T_k'$ is a SSYT with lattice reading word, as desired.
\qed

\end{proof}

\begin{example}
Let $\alpha = (2,2,1)$, $\lambda = (3,2)$, and $k=2$.
Consider the SSYT of shape $\lambda \oplus \Delta_\alpha$ with lattice reading word and two $1$'s in the first row of $\lambda$ shown on the left.
This gives rise to the SSYT of shape $\mathcal{S}(\lambda,\alpha^\triangleleft;2)$ with lattice reading word shown on the right.

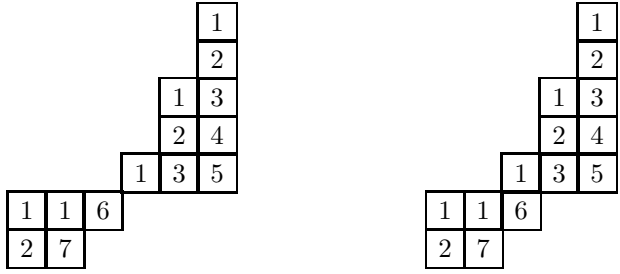
\setlength{\unitlength}{0.5mm}

\begin{picture}(0,100)(-30,10)

\put(70,90){\framebox(10,10)[c]{\textrm{ }$1$ }}
\put(70,80){\framebox(10,10)[c]{\textrm{ }$2$ } }
\put(70,70){\framebox(10,10)[c]{\textrm{ }$3$ } }
\put(70,60){\framebox(10,10)[c]{\textrm{ }$4$ } }
\put(70,50){\framebox(10,10)[c]{\textrm{ }$5$ } }

\put(60,70){\framebox(10,10)[c]{\textrm{ }$1$ } }
\put(60,60){\framebox(10,10)[c]{\textrm{ }$2$ } }
\put(60,50){\framebox(10,10)[c]{\textrm{ }$3$ } }

\put(50,50){\framebox(10,10)[c]{\textrm{ }$1$ } }

\put(20,40){\framebox(10,10)[c]{\textrm{ }$1$ } }
\put(30,40){\framebox(10,10)[c]{\textrm{ }$1$ } }
\put(40,40){\framebox(10,10)[c]{\textrm{ }$6$ } }

\put(20,30){\framebox(10,10)[c]{\textrm{ }$2$ } }
\put(30,30){\framebox(10,10)[c]{\textrm{ }$7$ } }

\put(170,90){\framebox(10,10)[c]{\textrm{ }$1$ }}
\put(170,80){\framebox(10,10)[c]{\textrm{ }$2$ } }
\put(170,70){\framebox(10,10)[c]{\textrm{ }$3$ } }
\put(170,60){\framebox(10,10)[c]{\textrm{ }$4$ } }
\put(170,50){\framebox(10,10)[c]{\textrm{ }$5$ } }

\put(160,70){\framebox(10,10)[c]{\textrm{ }$1$ } }
\put(160,60){\framebox(10,10)[c]{\textrm{ }$2$ } }
\put(160,50){\framebox(10,10)[c]{\textrm{ }$3$ } }

\put(150,50){\framebox(10,10)[c]{\textrm{ }$1$ } }

\put(130,40){\framebox(10,10)[c]{\textrm{ }$1$ } }
\put(140,40){\framebox(10,10)[c]{\textrm{ }$1$ } }
\put(150,40){\framebox(10,10)[c]{\textrm{ }$6$ } }

\put(130,30){\framebox(10,10)[c]{\textrm{ }$2$ } }
\put(140,30){\framebox(10,10)[c]{\textrm{ }$7$ } }

\end{picture}

\end{example}

\section{Sums of Fat Staircases}

Suppose a diagram $D= \rho / \kappa$ is such that
\[s_D = \sum_{\nu = \textrm{{\footnotesize fat stair}}} c_{\kappa \nu}^{\rho} s_\nu. \]
That is, for every $\nu$ that is not a fat staircase, we have $c_{\kappa \nu}^{\rho}=0$.
When this happens we say that $D$ is a \textit{sum of fat staircases}.
For each fat staircase $\nu$ we let $\alpha(\nu)$ be the composition such that $\delta_{\alpha(\nu)}= \nu$.
That is, $\alpha(\nu)_i$ is the number of parts of $\nu$ equal to $i$.
Then using these compositions we can rewrite $s_D$
as \[ s_D = \sum_{\nu = \textrm{{\footnotesize fat stair}}} c_{\kappa \nu}^{\rho} s_{\Delta_{ {\footnotesize \alpha(\nu)} }}  .\]

\begin{example}
Let $\rho = (4,3,3,3,3,3,3)$, $\kappa = (2,2,2,1,1)$, and $D =  \rho / \kappa$.
Then 
\begin{eqnarray*} 
s_D &=& s_{(4,3,2,2,1,1,1)} + s_{(3,3,3,2,1,1,1)} + s_{(3,3,2,2,2,1,1)} \\
&=&  s_{\Delta_{(3,2,1,1)}} + s_{\Delta_{(3,1,3)}} + s_{\Delta_{(2,3,2)}}.\\
\end{eqnarray*}

\end{example}

We now proceed to examine which skew diagrams $D$ are sums of fat staircases.
The following results are concerned with adding or removing columns to a skew diagram on the left or right side of the diagram. 

\begin{lemma}
\label{addonecol}
Let $c$ be a column and $D$ be a skew diagram that is not a sum of fat staircases and let $D_1 = D \odot_i c$  and $D_2 = c \odot_i D$ be obtained from $D$ by the addition of a single column.
Then neither $D_1$ nor $D_2$ is a sum of fat staircases.
\end{lemma}

\begin{proof}
We note that for $D_1$ and $D_2$ to be defined we require that $i \leq $ the length of the first column of $D$, $i \leq $ the length of the last column of $D$, and $i \leq $ the length of $c$. 

We shall begin by proving that $D_2$ is not a sum of fat staircases. 
Since $D$ is not a sum of fat staircases there is a SSYT $T$ of shape $D$ and lattice reading word whose content is $\nu$, where $\nu$ is not a fat staircase.
That is, $\nu_i \geq \nu_{i+1} +2$ for some $i$.

Let $n$ be the length of the column $c$. Thus $i \leq n$. 
We create a tableau $T'$ of shape $D_2$ and lattice reading word by filling $c$ with the numbers $1,2,\ldots,n$ and by filling the rest of $D_2$ as $D$ is filled in the tableau $T$.
Then $T'$ is semistandard and has a lattice reading word.
Further, $c(T') = \nu + (1^n)$.
Since $i \leq n$, there are two cases to consider. 
If $i < n$, comparing the $i$-th and $i+1$-th entries of this content gives
\[ c(T')_i = \nu_i +1 \geq \nu_{i+1} +3 = c(T')_{i+1} +2. \]
If $i = n$, comparing the $i$-th and $i+1$-th entries of this content gives
\[ c(T')_i = \nu_i +1 \geq \nu_{i+1} +3 = c(T')_{i+1} +3. \]

In both cases we see that $c(T')$ is not a fat staircase. Hence $D_2$ is not a sum of fat staircases. 

\

Then, since $s_{D^\circ}=s_D$ is also not a sum of fat staircases, the above shows that $c \odot_i D^\circ$ is not a sum of fat staircases.
Now, since \[ s_{D_1} = s_{D \odot_i c} = s_{(D \odot_i c)^\circ} = s_{c \odot_i D^\circ},\] we find that $D_1$ is not a sum of fat staircases. \qed
\end{proof}

\begin{corollary}
\label{notsumcor}
If $D$ is not a sum of fat staircases and $D'$ is obtained from $D$ by the addition any number of columns, then $D'$ is not a sum of fat staircases.
\end{corollary}

\begin{proof}
Let $D'$ be obtained from $D$ by adding, in order, the columns $c_1$, $c_2$, $\ldots$ $c_n$.
Let $D=D_0$ and, for each $i=1,\ldots, n$, let $D_i$ be the subdiagram of $D'$ consisting of $D$ and the columns $c_1$, $\ldots$, $c_i$.
Then using Lemma~\ref{addonecol} repeatedly we find that $D_i$ is not a sum of fat staircases for each $i=1,\ldots, n$.
Since $D_n = D'$, we are done. \qed
\end{proof}

\begin{corollary}
\label{issumcor}
If $D$ is a sum of fat staircases and $D'$ is a connected subdiagram of $D$ obtained by removing columns, then $D'$ is a sum of fat staircases.
\end{corollary}

\begin{proof}
This is precisely the contrapositive of Corollary~\ref{notsumcor} with the roles of $D$ and $D'$ reversed.\qed
\end{proof}

Therefore every diagram $D$ that is a sum of fat staircases can be viewed as a column extension of a smaller diagram $D'$ that is also a sum of fat staircases.
Knowing this, for a given diagram $D$ that is a sum of fat staircases, we now consider what length of columns can be added to $D$ and how much can a column overlap with $D$ if we wish the new diagram to also be a sum of fat staircases.
Towards this end, we have the following result.

\begin{lemma}
\label{lllllll}
Let $c$ be a column, $D = \rho / \mu$ be a sum of fat staircases, and $D'$ be given by either $D'= c  \odot_i D$  or $D' = D \odot_i c$.
If $D'$ is also a sum of fat staircases, then we have $l(c)+1 \not\in R_{\alpha(\nu),1}$ and $i + 1 \not\in R_{\alpha(\nu),1}$ for each $\nu$ with $c_{\mu \nu}^{\rho} \neq 0$.
\end{lemma}

\begin{proof}

We shall consider the case $D' = c \textrm{ } \odot_i D$.
As in the proof of Lemma~\ref{addonecol}, the second case $D'=D \odot_i c$ can be obtained by rotating diagrams by $180^\circ$.

Suppose $l(c)+1 \in R_{\alpha(\nu),1}$ for some $\nu$ with $c_{\mu \nu}^{\rho} \neq 0$. 
Since $c_{\mu \nu}^{\rho} \neq 0$, there is a SSYT $T$ of shape $D$ with lattice reading word and content $\nu$.
Now we create a tableau $T'$ of shape $D'$ by filling $c$ with the values $1,2, \ldots, l(c)$ and filling $D$ as in the tableau $T$.
Then $T'$ is clearly a SSYT with lattice reading word.
Further, $c(T') = \nu + (1^{l(c)})$.
Since $l(c)+1 \in R_{\alpha(\nu),1}$, we have $l(c) = \sum_{j=1}^{k} \alpha(\nu)_{n+1-j}$ for some $k$, where $\nu = (n^{\alpha(\nu)_n}, {n-1}^{\alpha(\nu)_{n-1}}, \ldots, 1^{\alpha(\nu)_1})$.
Thus 
\begin{eqnarray*}
c(T') &=& \nu + (1^{l(c)}) \\
      &=& (n^{\alpha(\nu)_n}, {n-1}^{\alpha(\nu)_{n-1}}, \ldots, {n+1-k}^{\alpha(\nu)_{n+1-k}},{n-k}^{\alpha(\nu)_{n-k}}  ,\ldots, 1^{\alpha(\nu)_1}) \\
      & \textrm{  } &  + (1^{\alpha(\nu)_n + \alpha(\nu)_{n-1} +\ldots \alpha(\nu)_{n+1-k}}) \\
      &=& ({n+1}^{\alpha(\nu)_n}, {n}^{\alpha(\nu)_{n-1}}, \ldots, {n+2-k}^{\alpha(\nu)_{n+1-k}},{n-k}^{\alpha(\nu)_{n-k}}  ,\ldots, 1^{\alpha(\nu)_1}). \\
\end{eqnarray*}
Thus $c(T')$ is not a fat staircase, and so $D'$ is not a sum of fat staircases.

Therefore, if $D'$ is a sum of fat staircases, then we require that $l(c)+1 \not\in R_{\alpha(\nu),1}$ for each $\nu$ with $c_{\mu \nu}^{\rho} \neq 0$.

\

Similarly, suppose that $i+1 \in R_{\alpha(\nu),1}$ for some $\nu$ with $c_{\mu \nu}^{\rho} \neq 0$. 
Since $c_{\mu \nu}^{\rho} \neq 0$, there is a SSYT $T$ of shape $D$ with lattice reading word and content $\nu$.
We create a tableau $T'$ of shape $D'$ by filling $c$ with the values $1,2, \ldots, i, l(\nu)+1, l(\nu)+2, \ldots, l(\nu) + l(c)-i$ and filling $D$ as in the tableau $T$.
Again, $T'$ is a SSYT with lattice reading word.
Further, $c(T') = \nu + (1^{i}) + (0^{l(\nu)}, 1^{l(c)-i})$.
Since $i+1 \in R_{\alpha(\nu),1}$, we have $i = \sum_{j=1}^{k} \alpha(\nu)_{n+1-j}$ for some $k$, where $\nu = (n^{\alpha(\nu)_n}, {n-1}^{\alpha(\nu)_{n-1}}, \ldots, 1^{\alpha(\nu)_1})$.
Thus 
\begin{eqnarray*}
c(T') &=& \nu + (1^{i})+ (0^{l(\nu)}, 1^{l(c)-i}) \\
      &=& (n^{\alpha(\nu)_n}, {n-1}^{\alpha(\nu)_{n-1}}, \ldots, {n+1-k}^{\alpha(\nu)_{n+1-k}},{n-k}^{\alpha(\nu)_{n-k}}  ,\ldots, 1^{\alpha(\nu)_1}) \\
      & \textrm{  } &  + (1^{\alpha(\nu)_n + \alpha(\nu)_{n-1} +\ldots \alpha(\nu)_{n+1-k}}) + (0^{l(\nu)}, 1^{l(c)-i}) \\
      &=& ({n+1}^{\alpha(\nu)_n}, {n}^{\alpha(\nu)_{n-1}}, \ldots, {n+2-k}^{\alpha(\nu)_{n+1-k}}, {n-k}^{\alpha(\nu)_{n-k}}  ,\ldots,  \\
      & &\textrm{ }\textrm{ }\textrm{ }\textrm{ }\textrm{ }\textrm{ }\textrm{ }\textrm{ }\textrm{ }\textrm{ }\textrm{ }\textrm{ }\textrm{ }\textrm{ }\textrm{ }\textrm{ }\textrm{ }\textrm{ }\textrm{ }\textrm{ }\textrm{ }\textrm{ }\textrm{ }\textrm{ }\textrm{ }\textrm{ }\textrm{ }\textrm{ }\textrm{ }\textrm{ }\textrm{ }\textrm{ }\textrm{ }\textrm{ }\textrm{ }\textrm{ }\textrm{ }\textrm{ }\textrm{ }\textrm{ }\textrm{ }\textrm{ }\textrm{ }\textrm{ }\textrm{ }\textrm{ }\textrm{ }\textrm{ }\textrm{ }\textrm{ }\textrm{ }\textrm{ }\textrm{ }\textrm{ }\textrm{ }\textrm{ }\textrm{ }\textrm{ }\textrm{ } \ldots,2^{\alpha(\nu)_2} ,1^{\alpha(\nu)_1 + l(c) -i}). \\
\end{eqnarray*}
Thus $c(T')$ is not a fat staircase, and so $D'$ is not a sum of fat staircases.

Therefore, if $D'$ is a sum of fat staircases, then we require that $i+1 \not\in R_{\alpha(\nu),1}$ for each $\nu$ with $c_{\mu \nu}^{\rho} \neq 0$. \qed
\end{proof}

It is important to note that although the conditions $l(c)+1 \not\in R_{\alpha(\nu),1}$ and $i + 1 \not\in R_{\alpha(\nu),1}$ are necessary in Lemma~\ref{lllllll}, they are by no means sufficient conditions. 
These restrictions can be used to eliminate candidates for column extensions. 
The next result continues this trend. 

\begin{theorem}
\label{tttttt}
If $D$ is a sum of fat staircases then the columns of $D$ have distinct lengths.
\end{theorem}

\begin{proof}
Let $D$ be a sum of fat staircases. Suppose $D$ has $n$ columns $c_1,\ldots,c_n$ and, for each $j$, let $m_j$ be the number of columns of length $j$.

We create a tableau $T$ of shape $D$ by filling each column $c_i$ with the entries $1,2,\ldots,l(c_i)$.
Then $T$ is a SSYT with lattice reading word and content $\nu = \sum_{i=1}^{n} (1^{l(c_i)})$.
Thus, for each $j$, $\nu_j$ is the number of columns of $D$ of length $\geq j$.
Therefore we have 
\begin{equation}
\label{fatafat}
\nu_j = \nu_{j+1} + m_j,
\end{equation} for each $j$.
Since $D$ is a sum of fat staircases, $\nu$ is a fat staircase. Therefore Equation~\ref{fatafat} implies that $0 \leq m_j \leq 1$ for each $j$.
In other words, the columns of $D$ have distinct lengths. \qed
\end{proof}

We now prove that the converse of Theorem~\ref{tttttt} holds in the case of a diagram with two columns.

\begin{lemma}
If $D$ is a connected diagram with two columns and these columns have distinct lengths, then $D$ is a sum of fat staircases.
\end{lemma}

\begin{proof}
Let $D$ be a connected diagram with two columns of distinct lengths.
Let $T$ be a SSYT of shape $D$ with lattice reading word and let $\nu$ be the content of $T$.
Since there are only two columns in $T$ and the entries of each column strictly increases we have $\nu_i \leq 2$ for each $i$.

Hence, if $D$ is not a sum of fat staircases then there must be a SSYT $T$ of shape $D$ and lattice reading word with content $\nu = (2^n)$, for some $n$.
Since the columns of $T$ strictly increase, this implies that both columns are of length $n$, contrary to assumption.
Therefore $D$ is a sum of fat staircases. \qed
\end{proof}

\

By using complements in a rectangle we now describe when a diagram of the form $\delta_\alpha / \lambda$ is a sum of fat staircases in the cases when $\lambda$ is either a single row or a single column.

\begin{theorem}
\label{rowcut}
Let $\alpha = (\alpha_1,\ldots, \alpha_n)$ be a composition. 
Then for $1 \leq m<n$ the skew diagram $\delta_\alpha / (m)$ is a sum of fat staircases if and only if $\alpha_j>1$ for each $j=1,2,\ldots, n-1$.
\end{theorem}

\begin{theorem}
\label{colcut}
Let $\alpha = (\alpha_1,\ldots, \alpha_n)$ be a composition. 
Then for $1 \leq m < |\alpha|$ the skew diagram $\delta_\alpha / (1^m)$ is a sum of fat staircases if and only if there is no $j$ such that $\alpha_j \leq m \leq |\alpha| - \alpha_{j+1}$. 
\end{theorem}

To prove these results we shall find it useful to make the following definition. 
Given a composition $\alpha = (\alpha_1,\alpha_2, \ldots, \alpha_n)$ and a width $w=n+k$, where $0 \leq k \leq 1$, we let 
\[ \alpha^r =  \left\{ \begin{array}{lcc}
            (\alpha_n, \alpha_{n-1},\ldots,\alpha_2, \alpha_1)  & \textrm{ if } & k=1 \\
            (\alpha_{n-1}, \alpha_{n-2},\ldots,\alpha_2, \alpha_1)  & \textrm{ if } & k=0 \\
         \end{array} \right. \] denote the \textit{reverse composition}.
With this definition we have ${\delta_\alpha}^c = \delta_{\alpha^r}$, where the complement is performed in the rectangle $(w^{|\alpha|})$.
We illustrate the two cases $k=0$ and $k=1$ below.

\
\

\

\

\

\

\

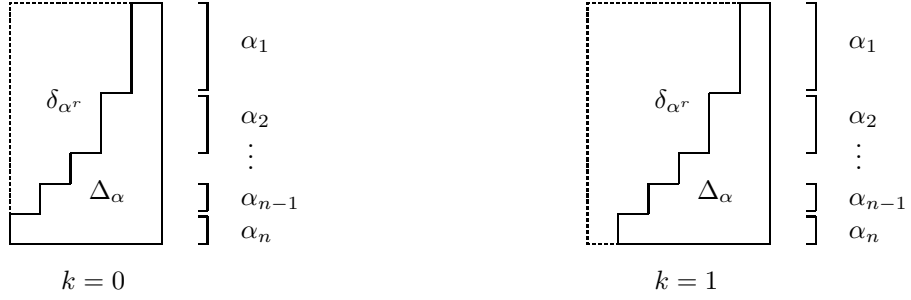
\setlength{\unitlength}{0.4mm}

\begin{picture}(100,60)(50,-15)

\put(282,05){$k=1$}
\put(87,05){$k=0$}

\put(282,65){$\delta_{\alpha^r}$}

\put(296,35){$\Delta_{\alpha}$}

\put(346,85){$\alpha_1$}
\put(346,60){$\alpha_2$}
\put(348,45){$\vdots$}
\put(148,45){$\vdots$}

\put(346,33){$\alpha_{n-1}$}
\put(346,22){$\alpha_n$}

\put(335,20){\line(0,1){9}}
\put(335,20){\line(-1,0){3}}
\put(335,29){\line(-1,0){3}}

\put(335,31){\line(0,1){9}}
\put(335,31){\line(-1,0){3}}
\put(335,40){\line(-1,0){3}}

\put(335,50){\line(0,1){19}}
\put(335,50){\line(-1,0){3}}
\put(335,69){\line(-1,0){3}}

\put(335,71){\line(0,1){29}}
\put(335,71){\line(-1,0){3}}
\put(335,100){\line(-1,0){3}}

\put(270,20){\line(1,0){50}}
\put(270,30){\line(1,0){10}}
\put(280,40){\line(1,0){10}}
\put(290,50){\line(1,0){10}}
\put(300,70){\line(1,0){10}}
\put(310,100){\line(1,0){10}}

\put(270,20){\line(0,1){10}}
\put(280,30){\line(0,1){10}}
\put(290,40){\line(0,1){10}}
\put(300,50){\line(0,1){20}}
\put(310,70){\line(0,1){30}}
\put(320,20){\line(0,1){80}}

\put(260,20){\dashbox{1}(60,80)[tl]{ }}

\put(82,65){$\delta_{\alpha^r}$}

\put(96,35){$\Delta_{\alpha}$}

\put(146,85){$\alpha_1$}
\put(146,60){$\alpha_2$}
\put(146,33){$\alpha_{n-1}$}
\put(146,22){$\alpha_n$}

\put(135,20){\line(0,1){9}}
\put(135,20){\line(-1,0){3}}
\put(135,29){\line(-1,0){3}}

\put(135,31){\line(0,1){9}}
\put(135,31){\line(-1,0){3}}
\put(135,40){\line(-1,0){3}}

\put(135,50){\line(0,1){19}}
\put(135,50){\line(-1,0){3}}
\put(135,69){\line(-1,0){3}}

\put(135,71){\line(0,1){29}}
\put(135,71){\line(-1,0){3}}
\put(135,100){\line(-1,0){3}}

\put(70,20){\line(1,0){50}}
\put(70,30){\line(1,0){10}}
\put(80,40){\line(1,0){10}}
\put(90,50){\line(1,0){10}}
\put(100,70){\line(1,0){10}}
\put(110,100){\line(1,0){10}}

\put(70,20){\line(0,1){10}}
\put(80,30){\line(0,1){10}}
\put(90,40){\line(0,1){10}}
\put(100,50){\line(0,1){20}}
\put(110,70){\line(0,1){30}}
\put(120,20){\line(0,1){80}}

\put(70,20){\dashbox{1}(50,80)[tl]{ }}

\end{picture}

Thus we have $l(\delta_{\alpha^r}) \leq l(\delta_{\alpha})$ and $w(\delta_{\alpha^r}) \leq w(\delta_{\alpha})$.
In particular, we have $|\alpha| = |\alpha^r| + (1-k)\alpha_n$.

\begin{proof} (of Theorem~\ref{rowcut})
The diagram $\delta_\alpha / (m)$ is contained in the rectangle $(n^{|\alpha|})$. 
Using Corollary~\ref{rectcor} we obtain $s_{\delta_\alpha / (m)} = c(s_{(m)} s_{\delta_{\alpha^r}})$. 
We note that $\alpha^r = (\alpha_{n-1}, \alpha_{n-2}, \ldots, \alpha_{1})$ since the width of the rectangle is $n$.

\

\

\

\

\

\

\

\

\

\

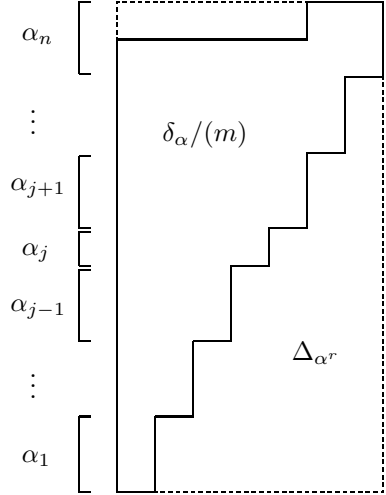
\setlength{\unitlength}{0.5mm}

\begin{picture}(100,60)(-50,-25)

\put(62,73){$\delta_{\alpha} / (m)$}

\put(96,15){$\Delta_{\alpha^r}$}
\put(22,60){$\alpha_{j+1}$}

\put(22,28){$\alpha_{j-1}$}
\put(25,100){$\alpha_n$}
\put(27,75){$\vdots$}
\put(25,43){$\alpha_j$}
\put(27,5){$\vdots$}

\put(25,-12){$\alpha_1$}

\put(40,20){\line(0,1){19}}
\put(40,20){\line(1,0){3}}
\put(40,39){\line(1,0){3}}

\put(40,40){\line(0,1){9}}
\put(40,40){\line(1,0){3}}
\put(40,49){\line(1,0){3}}

\put(40,50){\line(0,1){19}}
\put(40,50){\line(1,0){3}}
\put(40,69){\line(1,0){3}}

\put(40,91){\line(0,1){19}}
\put(40,91){\line(1,0){3}}
\put(40,110){\line(1,0){3}}

\put(40,-20){\line(0,1){20}}
\put(40,-20){\line(1,0){3}}
\put(40,0){\line(1,0){3}}

\put(70,20){\line(1,0){10}}
\put(80,40){\line(1,0){10}}
\put(90,50){\line(1,0){10}}
\put(100,70){\line(1,0){10}}
\put(110,90){\line(1,0){10}}
\put(50,100){\line(1,0){50}}

\put(100,110){\line(1,0){20}}
\put(100,100){\line(0,1){10}}

\put(60,0){\line(1,0){10}}

\put(70,0){\line(0,1){20}}

\put(80,20){\line(0,1){20}}
\put(90,40){\line(0,1){10}}
\put(100,50){\line(0,1){20}}
\put(110,70){\line(0,1){20}}
\put(120,90){\line(0,1){20}}
\put(50,-20){\line(0,1){120}}

\put(60,-20){\line(0,1){20}}
\put(50,-20){\line(1,0){10}}

\put(50,-20){\dashbox{1}(70,130)[tl]{ }}

\end{picture}

\

We can compute the product $s_{(m)} s_{\delta_{\alpha^r}}$ using Pieri's rule (\cite{stanley}). 
Thus we have 
\[s_{(m)} s_{\delta_{\alpha^r}} = \sum_\lambda s_\lambda,\]
where the sum is over all partitions $\lambda$ such that $\lambda / \delta_{\alpha^r}$ is a row strip with $m$ boxes. 
That is, the sum is over all partitions $\lambda$ such that the skew diagram $\lambda / \delta_{\alpha^r}$ has $m$ boxes, no two of which are in the same column. 

Since we have
\begin{eqnarray*}
s_{\delta_\alpha / (m)} &=& c(s_{(m)} s_{\delta_{\alpha^r}}) \\
&=& c(\sum_\lambda s_\lambda), \\
\end{eqnarray*}
we obtain
\begin{equation}
\label{pierirow}
s_{\delta_\alpha / (m)}= \sum_{\lambda \subseteq (n^{|\alpha|})} s_{\lambda^c}, 
\end{equation}
where the sum is over all partitions $\lambda \subseteq (n^{|\alpha|})$ such that the skew diagram $\lambda / \delta_{\alpha^r}$ has $m$ boxes, no two of which are in the same column.

\

Suppose there is a $j$, where $1 \leq j \leq n-1$, such that $\alpha_j =1$. 
We intend to show that $\delta_\alpha / (m)$ is not a sum of fat staircases. 
Namely, we will show that there is a term $s_{\nu}$ in Equation~\ref{pierirow} where $\nu$ is not a fat staircase. 

Since $\delta_{\alpha^r}$ is a fat staircase and we are constrained by $\lambda \subseteq (n^{|\alpha|})$, there are only $n$ possible positions to place the $m$ boxes. 
Namely, there is one position in the $r$-th row for each $r \in R_{\alpha^r,1}$. 
These $n$ possible positions are shown in the diagram below.

\

\

\

\

\

\

\

\

\

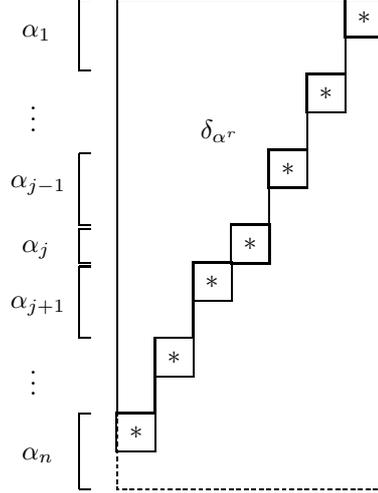
\setlength{\unitlength}{0.5mm}

\begin{picture}(100,60)(-40,-25)

\put(60,-10){\framebox(10,10){$*$}}
\put(70,10){\framebox(10,10){$*$}}
\put(80,30){\framebox(10,10){$*$}}
\put(90,40){\framebox(10,10){$*$}}
\put(100,60){\framebox(10,10){$*$}}
\put(110,80){\framebox(10,10){$*$}}
\put(120,100){\framebox(10,10){$*$}}

\put(82,73){$\delta_{\alpha^r}$}

\put(32,60){$\alpha_{j-1}$}

\put(32,28){$\alpha_{j+1}$}
\put(35,100){$\alpha_1$}
\put(37,75){$\vdots$}
\put(35,43){$\alpha_j$}
\put(37,5){$\vdots$}

\put(35,-12){$\alpha_n$}

\put(50,20){\line(0,1){19}}
\put(50,20){\line(1,0){3}}
\put(50,39){\line(1,0){3}}

\put(50,40){\line(0,1){9}}
\put(50,40){\line(1,0){3}}
\put(50,49){\line(1,0){3}}

\put(50,50){\line(0,1){19}}
\put(50,50){\line(1,0){3}}
\put(50,69){\line(1,0){3}}

\put(50,91){\line(0,1){19}}
\put(50,91){\line(1,0){3}}
\put(50,110){\line(1,0){3}}

\put(50,-20){\line(0,1){20}}
\put(50,-20){\line(1,0){3}}
\put(50,0){\line(1,0){3}}

\put(70,20){\line(1,0){10}}
\put(80,40){\line(1,0){10}}
\put(90,50){\line(1,0){10}}
\put(100,70){\line(1,0){10}}
\put(110,90){\line(1,0){10}}

\put(60,110){\line(1,0){60}}

\put(60,0){\line(1,0){10}}

\put(70,0){\line(0,1){20}}

\put(80,20){\line(0,1){20}}
\put(90,40){\line(0,1){10}}
\put(100,50){\line(0,1){20}}
\put(110,70){\line(0,1){20}}
\put(120,90){\line(0,1){20}}
\put(60,0){\line(0,1){110}}

\put(60,-20){\dashbox{1}(70,130)[tl]{ }}

\end{picture}

\

Since we must place $m \geq 1$ boxes, where $m<n$, we choose to place a box at the end of the row of length $\alpha_j$ and place the other $m-1 \leq n-2$ boxes in any of the other possible positions \textit{except for the position at the end of the first row of length $\alpha_{j+1}$}.
This gives a partition $\lambda \subseteq (n^{|\alpha|})$ such that $\lambda / \delta_{\alpha^r}$ has $m$ boxes, no two of which are in the same column.
Further, by the choice of placement of these boxes we see that $\lambda$ is not a fat staircase. 
Therefore $\lambda^c$ is also not a fat staircase. 
Since this term arises in Equation~\ref{pierirow}, this shows that $\delta_\alpha / (m)$ is not a sum of fat staircases.
This completes the first half of the proof.

\

Now suppose that $\alpha_j >1$ for each $j=1,2,\ldots,n-1$. 
We intend to show that $\delta_\alpha / (m)$ is a sum of fat staircases. 

As in the previous case, any term in Equation~\ref{pierirow} arises from placing the $m$ boxes among those $n$ possible positions.
Since $\alpha_j >1$ for each $j=1,2,\ldots,n-1$, it is clear that each of these placements results in a fat staircase.
Therefore $\delta_\alpha / (m)$ is a sum of fat staircases, as claimed. \qed
\end{proof}

\begin{example}
Consider the skew diagram $\delta_{(2,2,2)} / (2) = (3,3,2,2,1,1) / (2)$. 
Since $\alpha = (2,2,2)$ satisfies the hypotheses of Theorem~\ref{rowcut}, $\delta_{(2,2,2)} / (2)$ is a sum of fat staircases.
In particular, we have 
\begin{eqnarray*}
s_{\delta_{(2,2,2)} / (2)} &=& s_{(3,3,2,1,1)} + s_{(3,2,2,2,1)}+s_{(3,2,2,1,1,1)} \\
&=& s_{\delta_{(2,1,2)}} + s_{\delta_{(1,3,1)}} + s_{\delta_{(3,2,1)}}. \\
\end{eqnarray*}

\end{example}

\begin{proof} (of Theorem~\ref{colcut})
Again, the $\delta_\alpha / (1^m)$ is contained in the rectangle $(n^{|\alpha|})$. 
Using Corollary~\ref{rectcor} we obtain $s_{\delta_\alpha / (1^m)} = c(s_{(1^m)} s_{\delta_{\alpha^r}})$. 

\

\

\

\

\

\

\

\

\

\

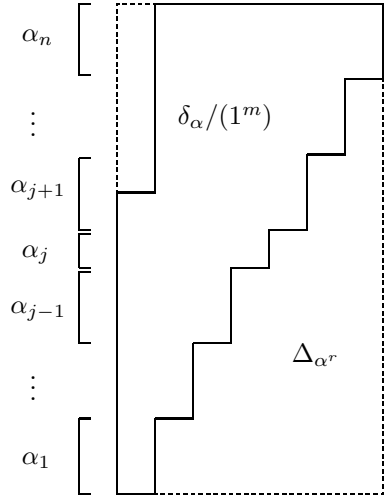
\setlength{\unitlength}{0.5mm}

\begin{picture}(100,60)(-50,-25)

\put(66,78){$\delta_{\alpha} / (1^m)$}

\put(96,15){$\Delta_{\alpha^r}$}
\put(22,60){$\alpha_{j+1}$}

\put(22,28){$\alpha_{j-1}$}
\put(25,100){$\alpha_n$}
\put(27,75){$\vdots$}
\put(25,43){$\alpha_j$}
\put(27,5){$\vdots$}

\put(25,-12){$\alpha_1$}

\put(40,20){\line(0,1){19}}
\put(40,20){\line(1,0){3}}
\put(40,39){\line(1,0){3}}

\put(40,40){\line(0,1){9}}
\put(40,40){\line(1,0){3}}
\put(40,49){\line(1,0){3}}

\put(40,50){\line(0,1){19}}
\put(40,50){\line(1,0){3}}
\put(40,69){\line(1,0){3}}

\put(40,91){\line(0,1){19}}
\put(40,91){\line(1,0){3}}
\put(40,110){\line(1,0){3}}

\put(40,-20){\line(0,1){20}}
\put(40,-20){\line(1,0){3}}
\put(40,0){\line(1,0){3}}

\put(70,20){\line(1,0){10}}
\put(80,40){\line(1,0){10}}
\put(90,50){\line(1,0){10}}
\put(100,70){\line(1,0){10}}
\put(110,90){\line(1,0){10}}
\put(60,110){\line(1,0){50}}

\put(100,110){\line(1,0){20}}

\put(60,0){\line(1,0){10}}

\put(70,0){\line(0,1){20}}

\put(80,20){\line(0,1){20}}
\put(90,40){\line(0,1){10}}
\put(100,50){\line(0,1){20}}
\put(110,70){\line(0,1){20}}
\put(120,90){\line(0,1){20}}
\put(50,-20){\line(0,1){80}}

\put(50,60){\line(1,0){10}}
\put(60,60){\line(0,1){50}}

\put(60,-20){\line(0,1){20}}
\put(50,-20){\line(1,0){10}}

\put(50,-20){\dashbox{1}(70,130)[tl]{ }}

\end{picture}

\

We can compute the product $s_{(1^m)} s_{\delta_{\alpha^r}}$ using the column version of Pieri's rule (\cite{stanley}). 
Thus we have 
\[s_{(1^m)} s_{\delta_{\alpha^r}} = \sum_\lambda s_\lambda,\]
where the sum is over all partitions $\lambda$ such that $\lambda / \delta_{\alpha^r}$ is a column strip with $m$ boxes. 
That is, the sum over all partitions $\lambda$ such that the skew diagram $\lambda / \delta_{\alpha^r}$ has $m$ boxes, no two of which are in the same row. 

Since
\begin{eqnarray*}
s_{\delta_\alpha / (1^m)} &=& c(s_{(1^m)} s_{\delta_{\alpha^r}}) \\
&=& c(\sum_\lambda s_\lambda), \\
\end{eqnarray*}
we obtain
\begin{equation}
\label{piericol}
s_{\delta_\alpha / (1^m)}= \sum_{\lambda \subseteq (n^{|\alpha|})} s_{\lambda^c}, 
\end{equation}
where the sum is over all partitions $\lambda \subseteq (n^{|\alpha|})$ such that the skew diagram $\lambda / \delta_{\alpha^r}$ has $m$ boxes, no two of which are in the same row.

\

Suppose there is a $j$ such that $\alpha_j \leq m \leq |\alpha|-\alpha_{j+1}$. 
Then, necessarily, $1 \leq j \leq n-1$. 
We intend to show that $\delta_\alpha / (1^m)$ is not a sum of fat staircases. 
Namely, we will show that there is a term $s_{\nu}$ in Equation~\ref{piericol} where $\nu$ is not a fat staircase. 

Since $\delta_{\alpha^r}$ is a fat staircase and we are constrained by $\lambda \subseteq (n^{|\alpha|})$, the only possible positions to place the $m$ boxes are among the $|\alpha|$ positions shown below. 
For the remainder of the proof we shall refer to the columns of this column strip of possible positions as \textit{possible columns}. 
We note that when placing boxes in these possible columns, we must start from the top in order for the resulting diagram to be weakly decreasing in row lengths.

\

\

\

\

\

\

\

\

\

\

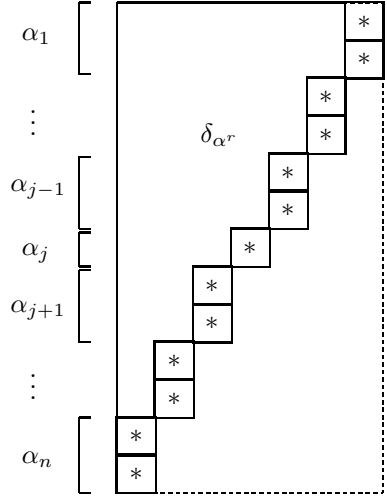
\setlength{\unitlength}{0.5mm}

\begin{picture}(100,60)(-40,-25)

\put(60,-10){\framebox(10,10){$*$}}
\put(60,-20){\framebox(10,10){$*$}}

\put(70,10){\framebox(10,10){$*$}}
\put(70,00){\framebox(10,10){$*$}}

\put(80,30){\framebox(10,10){$*$}}
\put(80,20){\framebox(10,10){$*$}}

\put(90,40){\framebox(10,10){$*$}}

\put(100,60){\framebox(10,10){$*$}}
\put(100,50){\framebox(10,10){$*$}}

\put(110,80){\framebox(10,10){$*$}}
\put(110,70){\framebox(10,10){$*$}}

\put(120,100){\framebox(10,10){$*$}}
\put(120,90){\framebox(10,10){$*$}}

\put(82,73){$\delta_{\alpha^r}$}

\put(32,60){$\alpha_{j-1}$}

\put(32,28){$\alpha_{j+1}$}
\put(35,100){$\alpha_1$}
\put(37,75){$\vdots$}
\put(35,43){$\alpha_j$}
\put(37,5){$\vdots$}

\put(35,-12){$\alpha_n$}

\put(50,20){\line(0,1){19}}
\put(50,20){\line(1,0){3}}
\put(50,39){\line(1,0){3}}

\put(50,40){\line(0,1){9}}
\put(50,40){\line(1,0){3}}
\put(50,49){\line(1,0){3}}

\put(50,50){\line(0,1){19}}
\put(50,50){\line(1,0){3}}
\put(50,69){\line(1,0){3}}

\put(50,91){\line(0,1){19}}
\put(50,91){\line(1,0){3}}
\put(50,110){\line(1,0){3}}

\put(50,-20){\line(0,1){20}}
\put(50,-20){\line(1,0){3}}
\put(50,0){\line(1,0){3}}

\put(70,20){\line(1,0){10}}
\put(80,40){\line(1,0){10}}
\put(90,50){\line(1,0){10}}
\put(100,70){\line(1,0){10}}
\put(110,90){\line(1,0){10}}

\put(60,110){\line(1,0){60}}

\put(60,0){\line(1,0){10}}

\put(70,0){\line(0,1){20}}

\put(80,20){\line(0,1){20}}
\put(90,40){\line(0,1){10}}
\put(100,50){\line(0,1){20}}
\put(110,70){\line(0,1){20}}
\put(120,90){\line(0,1){20}}
\put(60,0){\line(0,1){110}}

\put(60,-20){\dashbox{1}(70,130)[tl]{ }}

\end{picture}

Since we must place $m \geq \alpha_j$ boxes, where $m \leq |\alpha|-\alpha_{j+1}$, we choose to place a box at the end of each row of length $\alpha_j$ and place the other $m-\alpha_j \leq |\alpha|-\alpha_{j+1} -\alpha_j$ boxes in any of the other possible positions \textit{except for the positions at the end of the rows of length $\alpha_{j+1}$}.
This gives a partition $\lambda \subseteq (n^{|\alpha|})$ such that $\lambda / \delta_{\alpha^r}$ has $m$ boxes, no two of which are in the same row.
Further, by the choice of placement of these boxes we see that $\lambda$ is not a fat staircase. 
Therefore $\lambda^c$ is also not a fat staircase. 
Since this term arises in Equation~\ref{piericol}, this shows that $\delta_\alpha / (1^m)$ is not a sum of fat staircases.
This completes the first half of the proof.

\

Now suppose that there is no $j$ such that $\alpha_j \leq m \leq |\alpha| -\alpha_{j+1}$.
We intend to show that $\delta_\alpha / (1^m)$ is a sum of fat staircases. 

As in the previous case, any term in Equation~\ref{piericol} arises from placing the $m$ boxes among those $|\alpha|$ possible positions.
Since no $j$ satifies $\alpha_j \leq m \leq |\alpha| -\alpha_{j+1}$, there is no way to place these $m$ boxes so that a particular possible column is filled yet the possible column immediately to the left is empty. 
Since such a placement is the \textit{only} way to obtain a partition that is not a fat staircase, this shows that $\delta_\alpha / (1^m)$ is a sum of fat staircases. \qed
\end{proof}

\begin{example}
Consider the skew diagram $\delta_{(3,3,3)} / (1,1) = (3,3,3,2,2,2,1,1,1) / (1,1)$. 
Since $\alpha = (3,3,3)$ satisfies the hypotheses of Theorem~\ref{colcut}, $\delta_{(3,3,3)} /(1,1)$ is a sum of fat staircases.
In particular, we have 
\begin{eqnarray*}
s_{\delta_{(3,3,3)} / (1,1)} &=& s_{(3,3,3,2,2,2,1)} + s_{(3,3,3,2,2,1,1,1)}+s_{(3,3,3,2,1,1,1,1,1)} +s_{(3,3,2,2,2,2,1,1)} \\
&\textrm{}& +s_{(3,3,2,2,2,1,1,1,1)}+s_{(3,2,2,2,2,2,1,1,1)}\\
&=& s_{\delta_{(1,3,3)}} + s_{\delta_{(3,2,3)}} + s_{\delta_{(5,1,3)}}+ s_{\delta_{(2,4,2)}}+ s_{\delta_{(4,3,2)}}+ s_{\delta_{(3,5,1)}}. \\
\end{eqnarray*}

\end{example}

\section{Schur-Positivity}

For a sum of fat staircases $D = \rho / \kappa$ and a value $k \geq 0$, we consider the diagram $\mathcal{S}(\lambda,\mu, D ;k)$.
By assumption,  we have \[ s_D = \sum_{\nu = \textrm{{\footnotesize fat stair}}} c_{\kappa \nu}^{\rho} s_{\Delta_{ {\footnotesize \alpha(\nu)} }}  .\]
When we compute the skew Schur function  $s_{\mathcal{S}(\lambda,\mu, D ;k)}$ we notice that, since the subdiagram $D$ is strictly above the foundation $\lambda / \mu$, the contents that arise from the fillings of $D$ are precisely the fat stairs $\delta_{\alpha(\nu)}$, for each $\nu$ with $[s_\nu](s_D)\neq 0$.
For each of these fillings of $D$, we then need to extend them by filling the foundation $\lambda / \mu$ such that the resulting tableaux are SSYTx with lattice reading words.

Thus, when computing the Schur function of $\mathcal{S}(\lambda,\mu, D ;k)$, we are interested in the possible fillings of the following diagrams.

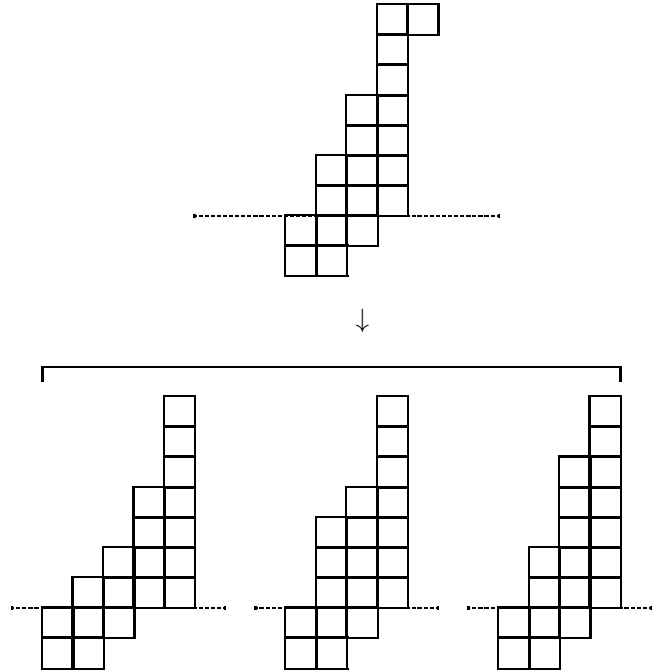
\setlength{\unitlength}{0.4mm}
\begin{picture}(100,60)(-30,0)

\put(113,-57){$\downarrow$}

\put(120,40){\framebox(10,10)[tl]{ }}
\put(130,40){\framebox(10,10)[tl]{ }}

\put(120,30){\framebox(10,10)[tl]{ }}

\put(120,20){\framebox(10,10)[tl]{ }}

\put(110,10){\framebox(10,10)[tl]{ }}
\put(120,10){\framebox(10,10)[tl]{ }}

\put(110,0){\framebox(10,10)[tl]{ }}
\put(120,0){\framebox(10,10)[tl]{ }}

\put(100,-10){\framebox(10,10)[tl]{ }}
\put(110,-10){\framebox(10,10)[tl]{ }}
\put(120,-10){\framebox(10,10)[tl]{ }}

\put(100,-20){\framebox(10,10)[tl]{ }}
\put(110,-20){\framebox(10,10)[tl]{ }}
\put(120,-20){\framebox(10,10)[tl]{ }}

\put(60,-20){\dashbox{1}(100,00)[tl]{ }}

\put(90,-30){\framebox(10,10)[tl]{ }}
\put(100,-30){\framebox(10,10)[tl]{ }}
\put(110,-30){\framebox(10,10)[tl]{ }}

\put(90,-40){\framebox(10,10)[tl]{ }}
\put(100,-40){\framebox(10,10)[tl]{ }}

\put(50,-90){\framebox(10,10)[tl]{ }}
\put(50,-100){\framebox(10,10)[tl]{ }}
\put(50,-110){\framebox(10,10)[tl]{ }}
\put(50,-120){\framebox(10,10)[tl]{ }}
\put(50,-130){\framebox(10,10)[tl]{ }}
\put(50,-140){\framebox(10,10)[tl]{ }}
\put(50,-150){\framebox(10,10)[tl]{ }}

\put(40,-120){\framebox(10,10)[tl]{ }}
\put(40,-130){\framebox(10,10)[tl]{ }}
\put(40,-140){\framebox(10,10)[tl]{ }}
\put(40,-150){\framebox(10,10)[tl]{ }}

\put(30,-140){\framebox(10,10)[tl]{ }}
\put(30,-150){\framebox(10,10)[tl]{ }}

\put(20,-150){\framebox(10,10)[tl]{ }}

\put(0,-150){\dashbox{1}(70,00)[tl]{ }}

\put(80,-150){\dashbox{1}(60,00)[tl]{ }}
\put(150,-150){\dashbox{1}(60,00)[tl]{ }}

\put(90,-160){\framebox(10,10)[tl]{ }}
\put(100,-160){\framebox(10,10)[tl]{ }}
\put(110,-160){\framebox(10,10)[tl]{ }}

\put(90,-170){\framebox(10,10)[tl]{ }}
\put(100,-170){\framebox(10,10)[tl]{ }}

\put(10,-160){\framebox(10,10)[tl]{ }}
\put(20,-160){\framebox(10,10)[tl]{ }}
\put(30,-160){\framebox(10,10)[tl]{ }}

\put(10,-170){\framebox(10,10)[tl]{ }}
\put(20,-170){\framebox(10,10)[tl]{ }}

\put(160,-160){\framebox(10,10)[tl]{ }}
\put(170,-160){\framebox(10,10)[tl]{ }}
\put(180,-160){\framebox(10,10)[tl]{ }}

\put(160,-170){\framebox(10,10)[tl]{ }}
\put(170,-170){\framebox(10,10)[tl]{ }}

\put(10,-70){\line(1,0){190}}

\put(10,-70){\line(0,-1){5}}
\put(200,-70){\line(0,-1){5}}

\put(120,-90){\framebox(10,10)[tl]{ }}
\put(120,-100){\framebox(10,10)[tl]{ }}
\put(120,-110){\framebox(10,10)[tl]{ }}
\put(120,-120){\framebox(10,10)[tl]{ }}
\put(120,-130){\framebox(10,10)[tl]{ }}
\put(120,-140){\framebox(10,10)[tl]{ }}
\put(120,-150){\framebox(10,10)[tl]{ }}

\put(110,-120){\framebox(10,10)[tl]{ }}
\put(110,-130){\framebox(10,10)[tl]{ }}
\put(110,-140){\framebox(10,10)[tl]{ }}
\put(110,-150){\framebox(10,10)[tl]{ }}

\put(100,-130){\framebox(10,10)[tl]{ }}
\put(100,-140){\framebox(10,10)[tl]{ }}
\put(100,-150){\framebox(10,10)[tl]{ }}

\put(190,-90){\framebox(10,10)[tl]{ }}
\put(190,-100){\framebox(10,10)[tl]{ }}
\put(190,-110){\framebox(10,10)[tl]{ }}
\put(190,-120){\framebox(10,10)[tl]{ }}
\put(190,-130){\framebox(10,10)[tl]{ }}
\put(190,-140){\framebox(10,10)[tl]{ }}
\put(190,-150){\framebox(10,10)[tl]{ }}

\put(180,-110){\framebox(10,10)[tl]{ }}
\put(180,-120){\framebox(10,10)[tl]{ }}
\put(180,-130){\framebox(10,10)[tl]{ }}
\put(180,-140){\framebox(10,10)[tl]{ }}
\put(180,-150){\framebox(10,10)[tl]{ }}

\put(170,-140){\framebox(10,10)[tl]{ }}
\put(170,-150){\framebox(10,10)[tl]{ }}

\end{picture}

\

\

\

\

\

\

\

\

\

\

\

\

\

\

\

\

\

\

\

The next result summarizes the relationship between these fillings.

\begin{theorem}
\label{sumoffattheorem}
Suppose a diagram $D= \rho / \kappa$ is such that
\[s_D = \sum_{\nu = \textrm{{\footnotesize fat stair}}} c_{\kappa \nu}^{\rho} s_\nu. \]
Then for each partitions $\lambda$ and $\mu$ with $\mu \subseteq \lambda$, and for each $k \geq 0$ we have
\[ s_{\mathcal{S}(\lambda, \mu , D;k)} \leq_s \sum_{\nu} c_{\kappa \nu}^{\rho}  s_{\mathcal{S}(\lambda, \mu , \alpha(\nu)^\triangleleft ;k)}. \]
\end{theorem}

\begin{proof}
Since $D$ is a sum of fat staircases, the content of each SSYT of shape $D$ with lattice reading word is some fat staircase $\nu$.
To prove the identity, we consider the map that takes a SSYT $T$ of shape $\mathcal{S}(\lambda, \mu , D;k)$ with lattice reading word and $c(D)=\nu$ to the tableau $T'$ of shape $\mathcal{S}(\lambda, \mu , \alpha(\nu)^\triangleleft ;k)$ obtained by filling the copy of $\Delta_{\alpha(\nu)}$ with content $\nu$ and filling the copy of $\lambda / \mu$ identically to its filling in $T$.

We claim that each tableau $T'$ is also a SSYT with lattice reading word.
Given such a tableau $T'$, consider the corresponding tableau $\mathcal{T}'$ of shape $\lambda / \mu \bigoplus \Delta_{\alpha(\nu)}$.
Then $\mathcal{T}'$ is a SSYT since both of the subtableaux of shape $\lambda / \mu$ and $\Delta_{\alpha(\nu)}$ are semistandard.
Further $\mathcal{T}'$ has a lattice reading word since checking the lattice condition within the subtableau of shape $\Delta_{\alpha(\nu)}$ is trivial, and then, by the construction of $T'$, the lattice condition after this point simply becomes identical to the lattice condition for $T$.
Therefore $\mathcal{T}'$ is a SSYT with lattice reading word.
Further, the first row of the subtableau $\lambda / \mu$ of $\mathcal{T}'$ contains at most $k$ 1's since $T$ was a SSYT. 
Hence, by Lemma~\ref{kfatjoinlemma}, $T'$ is also a SSYT with lattice reading word, exactly as claimed.

Thus, for each SSYT $T$ of shape $\mathcal{S}(\lambda, \mu , D;k)$ with lattice reading word and $c(D)=\nu$ we obtain a SSYT $T'$ of shape $\mathcal{S}(\lambda, \mu , \alpha(\nu)^\triangleleft ;k)$.
Now fix an image $T'$ and suppose that $T_1 \neq T_2$ are both SSYT of shape $\mathcal{S}(\lambda, \mu , D;k)$ with lattice reading word such that $T_1' = T' =T_2'$.
Let $\mathcal{S}(\lambda, \mu , \alpha(\nu)^\triangleleft ;k)$ be the shape of $T_1' =T'= T_2'$.
Then the filling of the foundation $\lambda / \mu$ in $T_1$ is the same as the filling of the foundation $\lambda / \mu$ in $T_2$, and the content of the subdiagram $D$ is $\nu$ for both $T_1$ and $T_2$.
Since there are only $c_{\kappa \nu}^{\rho}$ SSYT of shape $D$ with lattice reading word and content $\nu$, there are at most $c_{\kappa \nu}^{\rho}$ distinct SSYT of shape $\mathcal{S}(\lambda, \mu , D;k)$ with lattice reading word that could map to $T'$.
This implies that \[ s_{\mathcal{S}(\lambda, \mu , D;k)} \leq_s \sum_{\nu} c_{\kappa \nu}^{\rho}  s_{\mathcal{S}(\lambda, \mu , \alpha(\nu)^\triangleleft ;k)}, \]
as we desired.\qed

\end{proof}

\begin{example} Let $D = (2,2,2,2,1) / (1,1)$, $\lambda = (2,2)$, $\mu = \emptyset$ , and $k=1$.
Then \[ s_D = s_{(2,2,2,1)} + s_{(2,2,1,1,1)} = s_{\Delta_{(1,3)}} + s_{\Delta_{(3,2)}}\] is a sum of fat staircases.
Here we display $D$, $\Delta_{(1,3)}$, and $\Delta_{(3,2)}$.

\setlength{\unitlength}{0.4mm}

\begin{picture}(0,90)(-80,40)

\put(-35,55){$D$}
\put(65,55){$\Delta_{(1,3)}$}
\put(165,55){$\Delta_{(3,2)}$}

\put(-30,110){\framebox(10,10)[c]{ }}
\put(-30,100){\framebox(10,10)[c]{ } }
\put(-30,90){\framebox(10,10)[c]{ }}
\put(-30,80){\framebox(10,10)[c]{ } }

\put(-40,90){\framebox(10,10)[c]{ }}
\put(-40,80){\framebox(10,10)[c]{ } }
\put(-40,70){\framebox(10,10)[c]{ } }

\put(70,100){\framebox(10,10)[c]{ }}
\put(70,90){\framebox(10,10)[c]{ }}
\put(70,80){\framebox(10,10)[c]{ } }
\put(70,70){\framebox(10,10)[c]{ } }

\put(60,90){\framebox(10,10)[c]{ }}
\put(60,80){\framebox(10,10)[c]{ } }
\put(60,70){\framebox(10,10)[c]{ } }

\put(170,110){\framebox(10,10)[c]{ }}
\put(170,100){\framebox(10,10)[c]{ }}
\put(170,90){\framebox(10,10)[c]{ }}
\put(170,80){\framebox(10,10)[c]{ } }
\put(170,70){\framebox(10,10)[c]{ } }

\put(160,80){\framebox(10,10)[c]{ } }
\put(160,70){\framebox(10,10)[c]{ } }

\end{picture}

We are now interested in the diagrams $\mathcal{S}(\lambda, D;1)$, $\mathcal{S}(\lambda, (1,3)^\triangleleft;1)$, and $\mathcal{S}(\lambda, (3,2)^\triangleleft;1)$ for $\lambda = (2,2)$.

\setlength{\unitlength}{0.4mm}

\begin{picture}(0,110)(-80,20)

\put(-65,35){$\mathcal{S}(\lambda, D;1)$}
\put(35,35){$\mathcal{S}(\lambda, (1,3)^\triangleleft;1)$}
\put(135,35){$\mathcal{S}(\lambda, (3,2)^\triangleleft;1)$}

\put(-30,110){\framebox(10,10)[c]{ }}
\put(-30,100){\framebox(10,10)[c]{ } }
\put(-30,90){\framebox(10,10)[c]{ }}
\put(-30,80){\framebox(10,10)[c]{ } }

\put(-40,90){\framebox(10,10)[c]{ }}
\put(-40,80){\framebox(10,10)[c]{ } }
\put(-40,70){\framebox(10,10)[c]{ } }

\put(-40,60){\framebox(10,10)[c]{ } }
\put(-40,50){\framebox(10,10)[c]{ } }
\put(-50,60){\framebox(10,10)[c]{ } }
\put(-50,50){\framebox(10,10)[c]{ } }

\put(70,100){\framebox(10,10)[c]{ }}
\put(70,90){\framebox(10,10)[c]{ }}
\put(70,80){\framebox(10,10)[c]{ } }
\put(70,70){\framebox(10,10)[c]{ } }

\put(60,90){\framebox(10,10)[c]{ }}
\put(60,80){\framebox(10,10)[c]{ } }
\put(60,70){\framebox(10,10)[c]{ } }

\put(60,60){\framebox(10,10)[c]{ } }
\put(60,50){\framebox(10,10)[c]{ } }
\put(50,60){\framebox(10,10)[c]{ } }
\put(50,50){\framebox(10,10)[c]{ } }

\put(170,110){\framebox(10,10)[c]{ }}
\put(170,100){\framebox(10,10)[c]{ }}
\put(170,90){\framebox(10,10)[c]{ }}
\put(170,80){\framebox(10,10)[c]{ } }
\put(170,70){\framebox(10,10)[c]{ } }

\put(160,80){\framebox(10,10)[c]{ } }
\put(160,70){\framebox(10,10)[c]{ } }

\put(160,60){\framebox(10,10)[c]{ } }
\put(160,50){\framebox(10,10)[c]{ } }
\put(150,60){\framebox(10,10)[c]{ } }
\put(150,50){\framebox(10,10)[c]{ } }

\end{picture}

\noindent We can compute that 
\begin{eqnarray*}  s_{\mathcal{S}(\lambda,D;1)} &=& s_{(3,3,2,2,1)} + s_{(3,3,2,1,1,1)} + s_{(3,3,1,1,1,1,1)}+ s_{(3,2,2,2,2)}\\
&+&  s_{(3,2,2,2,1,1)} + s_{(3,2,2,1,1,1,1)}+ s_{(2,2,2,2,2,1)} + s_{(2,2,2,2,1,1,1)}\\
\end{eqnarray*}
and 
\begin{eqnarray*} s_{\mathcal{S}(\lambda,(1,3)^\triangleleft;1)}+s_{\mathcal{S}(\lambda,(3,2)^\triangleleft;1)} &=& 2 s_{(3,3,2,2,1)} + 2 s_{(3,3,2,1,1,1)} + s_{(3,3,1,1,1,1,1)}+ s_{(3,2,2,2,2)} \\
&+& 2 s_{(3,2,2,2,1,1)}+ s_{(3,2,2,1,1,1,1)}+ s_{(2,2,2,2,2,1)} + s_{(2,2,2,2,1,1,1)}, \\
\end{eqnarray*}
so we see that  \[ s_{\mathcal{S}(\lambda,D;1)} \leq_s s_{\mathcal{S}(\lambda,(1,3)^\triangleleft;1)}+s_{\mathcal{S}(\lambda,(3,2)^\triangleleft;1)}, \]
as predicted by Theorem~\ref{sumoffattheorem}.
\end{example}

\

\begin{corollary}
Suppose a diagram $D= \rho / \kappa$ is such that
\[s_D = \sum_{\nu = \textrm{{\footnotesize fat stair}}} c_{\kappa \nu}^{\rho} s_\nu. \]
Given a partition $\lambda$ and $k \geq 0$, let $\beta$ and $\theta$ be the partitions such that $\mathcal{S}(\lambda,\mu,  D;k) = \beta / \theta$. 
Then 
\[c(s_{\theta} s_{\beta^c}) \leq_s \sum_{\nu } c_{\kappa \nu}^{\rho}  s_{\mathcal{S}(\lambda, \mu, \alpha(\nu);k)}, \]
where $\beta^c$ is the complement of $\beta$ in the rectangle $({w(\beta / \theta)}^{l(\beta / \theta)})$.
\end{corollary}

\begin{proof}
From Corollary~\ref{rectcor} we have
\[ s_{\beta / \theta} = c(s_\theta s_{\beta^c}). \]
From Theorem~\ref{sumoffattheorem} we have
\[ s_{\mathcal{S}(\lambda, \mu, D;k)} \leq_s \sum_{\nu} c_{\kappa \nu}^{\rho}  s_{\mathcal{S}(\lambda,\mu, \alpha(\nu)^\triangleleft ;k)}. \]
Therefore 
\begin{eqnarray*}
c(s_\theta s_{\beta^c}) &=& s_{\beta / \theta} \\
&=& s_{\mathcal{S}(\lambda, \mu, D;k)} \\
&\leq_s& \sum_{\nu} c_{\kappa \nu}^{\rho}  s_{\mathcal{S}(\lambda, \mu, \alpha(\nu)^\triangleleft ;k)}. \qed\\
\end{eqnarray*}

\end{proof}

\

We shall finish this section with a result for the difference $s_{\mathcal{S}(\lambda^t, D;1)} - s_{\mathcal{S}(\lambda, D;1)}$, where $D$ is a sum of fat staircases and $\lambda$ is a single row. 
However, we begin by looking at the corresponding result using a fat staircase $\Delta_\alpha$.

\begin{theorem}
\label{1theorem}
If $\lambda$ is a partition with a single part, $0 \leq k \leq 1$, and $\lambda_1 - k \leq l(\alpha)$, then we have 
\[ s_{\mathcal{S}(\lambda^{t}, \alpha^\triangleleft;k)}-s_{\mathcal{S}(\lambda,\alpha^\triangleleft;k)} \geq_s 0. \] 
\end{theorem}

\begin{proof} 

\

\noindent

To prove the Schur-positivity of $s_{\mathcal{S}(\lambda^{t},\alpha^\triangleleft;k)}-s_{\mathcal{S}(\lambda,\alpha^\triangleleft;k)}$ we show that any SSYT $\mathcal{T}_1$ of shape $\mathcal{S}(\lambda,\alpha^\triangleleft;k)$ with lattice reading word gives rise to a SSYT $\mathcal{T}_2$ of shape $\mathcal{S}(\lambda^{t},\alpha^\triangleleft;k)$ with lattice reading word and the same content. 
We then show that we can recover $\mathcal{T}_1$ from $\mathcal{T}_2$.

Consider any SSYT $\mathcal{T}_1$ of shape $\mathcal{S}(\lambda,\alpha^\triangleleft;k)$ with lattice reading word. 
Lemma~\ref{kfatfirstrowlemma} implies that the foundation $\lambda$ contains a strictly increasing sequence $a_1<a_2<\ldots < a_{\lambda_1}$ where each $a_i \in R_{\alpha,k}$ and $a_{\lambda_1} \leq |\alpha|+1$ by the lattice condition. 

Since $a_1<a_2<\ldots < a_{\lambda_1}$, we can create a SSYT of shape $\lambda^{t}$ by transposing the entries of the foundation $\lambda$. 
We claim that the corresponding tableau $\mathcal{T}_2$ of shape $\mathcal{S}(\lambda^{t},\alpha^\triangleleft;k)$ is semistandard.
Since both $\Delta_\alpha$ and $\lambda^t$ are semistandard, we need only inspect where the columns of these subtableaux overlap.
These two diagrams only overlap when $k=0$, in which case the entry $a_1$ appears directly below some entry $x$.
However, when $k=0$, the entry $a_1$ appears below this same $x$ in $\mathcal{T}_1$. 
Since $\mathcal{T}_1$ is semistandard, this implies that $a_1>x$. 
Therefore $\mathcal{T}_2$ is semistandard. 
Furthermore, we may recover $\mathcal{T}_1$ from $\mathcal{T}_2$ by transposing the entries of the foundation $\lambda^{t}$ of $\mathcal{T}_2$.

Since every value in $\lambda^{t}$ is from $R_{\alpha,k}$, it is clear that $\mathcal{T}_2$ has a lattice reading word. 
This completes the proof that $s_{\mathcal{S}(\lambda^{t},\alpha^\triangleleft;k)}-s_{\mathcal{S}(\lambda,\alpha^\triangleleft;k)} \geq_s 0$.\qed

\end{proof}

\begin{theorem} 
\label{sumofdiff}
Let $D = \rho / \kappa$ be such that
\[ s_D= \sum_{\nu = \textrm{fat staircase}} c_{\kappa \nu}^{\rho} s_\nu,\] and $\lambda$ be a single row such that $\lambda_1-1 \leq$ the length of the last row of $D$. Then
\[s_{\mathcal{S}(\lambda^t, D;1)} - s_{\mathcal{S}(\lambda, D;1)} \geq_s \sum_{\nu} c_{\kappa \nu}^{\rho}  (s_{\mathcal{S}(\lambda^t,  \alpha(\nu)^{\triangleleft};1)}- s_{\mathcal{S}(\lambda,  \alpha(\nu)^{\triangleleft};1)}) \geq_s 0.\]
\end{theorem}

\begin{proof}
Applying Theorem~\ref{sumoffattheorem} shows that 
\[ s_{\mathcal{S}(\lambda, D;1)} \leq_s \sum_{\nu} c_{\kappa \nu}^{\rho} s_{\mathcal{S}(\lambda,  \alpha(\nu)^{\triangleleft};1)}, \]
so we have 
\begin{equation}
\label{bit1}
-s_{\mathcal{S}(\lambda, D;1)} \geq_s -\sum_{\nu} c_{\kappa \nu}^{\rho} s_{\mathcal{S}(\lambda,  \alpha(\nu)^{\triangleleft};1)}. 
\end{equation}
We now intend to show that 
\begin{equation}
\label{bit2} 
s_{\mathcal{S}(\lambda^t, D;1)} = \sum_{\nu} c_{\kappa \nu}^{\rho} s_{\mathcal{S}(\lambda^t,  \alpha(\nu)^{\triangleleft};1)}. 
\end{equation}
Then, adding Equation~\ref{bit1} and Equation~\ref{bit2} will give that
\[s_{\mathcal{S}(\lambda^t, D;1)} - s_{\mathcal{S}(\lambda, D;1)} \geq_s \sum_{\nu} c_{\kappa \nu}^{\rho}  (s_{\mathcal{S}(\lambda^t,  \alpha(\nu)^{\triangleleft};1)}- s_{\mathcal{S}(\lambda,  \alpha(\nu)^{\triangleleft};1)}).\]

\

To prove Equation~\ref{bit2}, we first note that Theorem~\ref{sumoffattheorem} shows that 
\[ s_{\mathcal{S}(\lambda^t, D;1)} \leq_s \sum_{\nu} c_{\mu \nu}^{\rho} s_{\mathcal{S}(\lambda^t,  \alpha(\nu)^{\triangleleft};1)}. \]
We now recall how the proof of Theorem~\ref{sumoffattheorem} proceeded.
The proof looked at the map that took a SSYT $T$ of shape $\mathcal{S}(\lambda^t,D;1)$ with lattice reading word and $c(D)=\nu$ to the tableau $T'$ of shape $\mathcal{S}(\lambda^t,\alpha(\nu)^\triangleleft;1)$ whose foundation $\lambda^t$ was filled identically as in $T$. Then for each such image $T'$ the proof showed that 
\begin{enumerate}
\item $T'$ was a SSYT with lattice reading word and $c(T')=c(T)$, and 
\item there are at most $c_{\kappa \nu}^{\rho}$ distinct SSYT of shape $\mathcal{S}(\lambda^t,  D;1)$ that could map to $T'$.
\end{enumerate}
Further, the $c_{\kappa \nu}^{\rho}$ possible tableau of shape $\mathcal{S}(\lambda^t,  D;1)$ mentioned in 2 were the tableaux of shape $\mathcal{S}(\lambda^t,  D;1)$ where the filling of $\lambda^t$ was identical to the filling of $\lambda^t$ in $T'$ and the filling of $D$ was one of the $c_{\kappa \nu}^{\rho}$ semistandard fillings of $D$ with lattice reading word and content $\nu$.
We now check that each of these preimages of $T'$ \textit{is} a SSYT of shape $\mathcal{S}(\lambda^t,  D;1)$ with lattice reading word.

Let $\mathcal{T}$ be a preimage of $T'$. 
Then the subtableau of $\mathcal{T}$ of shape $D$ is semistandard and has lattice reading word since, as just mentioned, it is one of the $c_{\kappa \nu}^{\rho}$ semistandard fillings of $D$ with lattice reading word and content $\nu$. 
We also have that the foundation $\lambda^t$ of $\mathcal{T}$ is semistandard since it is filled indentically to the foundation of $T'$, which was semistandard.
Further, since $\lambda^t$ is a single column and $k=1$, the foundation of $\mathcal{S}(\lambda^t,  D;1)$ does not appear directly below any box of $D$.
Therefore $\mathcal{T}$ is a SSYT.
Also, $\mathcal{T}$ has a lattice reading word since $D$ has a lattice reading word and, since we have $c(D)=c(\Delta_{\alpha(\nu)})$, after reading $D$ the lattice restrictions on $\mathcal{T}$ are identical to the lattice restrictions on $T'$, which has a lattice reading word.

Hence, we have shown that each of these $c_{\kappa \nu}^{\rho}$ preimages of $T'$ is a SSYT of shape $\mathcal{S}(\lambda^t,  D;1)$ with lattice reading word.
This proves Equation~\ref{bit2}. That is, 
\[ s_{\mathcal{S}(\lambda^t, D;1)} = \sum_{\nu} c_{\kappa \nu}^{\rho} s_{\mathcal{S}(\lambda^t,  \alpha(\nu)^{\triangleleft};1)}. \]
We may now combine Equation~\ref{bit1} and Equation~\ref{bit2} to obtain
\[s_{\mathcal{S}(\lambda^t, D;1)} - s_{\mathcal{S}(\lambda, D;1)} \geq_s \sum_{\nu} c_{\kappa \nu}^{\rho}  (s_{\mathcal{S}(\lambda^t,  \alpha(\nu)^{\triangleleft};1)}- s_{\mathcal{S}(\lambda,  \alpha(\nu)^{\triangleleft};1)}).\]
Now, using Theorem~\ref{1theorem} on each pair $s_{\mathcal{S}(\lambda^t,  \alpha(\nu)^{\triangleleft};1)}$, $s_{\mathcal{S}(\lambda,  \alpha(\nu)^{\triangleleft};1)}$ gives
\[ s_{\mathcal{S}(\lambda^t,  \alpha(\nu)^{\triangleleft};1)}- s_{\mathcal{S}(\lambda,  \alpha(\nu)^{\triangleleft};1)} \geq_s 0,\]
for each fat staircase $\alpha(\nu)$.
Therefore, we now have that
\[s_{\mathcal{S}(\lambda^t, D;1)} - s_{\mathcal{S}(\lambda, D;1)} \geq_s \sum_{\nu} c_{\kappa \nu}^{\rho}  (s_{\mathcal{S}(\lambda^t,  \alpha(\nu)^{\triangleleft};1)}- s_{\mathcal{S}(\lambda,  \alpha(\nu)^{\triangleleft};1)}) \geq_s  0. \qed \] 
\end{proof}

\begin{example} 
Once again we consider the example with $\rho = (4,3,3,3,3,3,3)$, $\kappa = (2,2,2,1,1)$, and $D =  \rho / \kappa$.
We have previous seen that $D$ is a sum of fat staircases.
Namely, we have 
\begin{eqnarray*} 
s_D &=& s_{(4,3,2,2,1,1,1)} + s_{(3,3,3,2,1,1,1)} + s_{(3,3,2,2,2,1,1)} \\
&=&  s_{\Delta_{(3,2,1,1)}} + s_{\Delta_{(3,1,3)}} + s_{\Delta_{(2,3,2)}}.\\
\end{eqnarray*}

We now wish to apply Theorem~\ref{sumofdiff} to this diagram $D$ using $\lambda=(3)$. 
We are interested in appending $\lambda$ to $D$ and each of $\Delta_{(3,2,1,1)}$, $\Delta_{(3,1,3)}$, and $\Delta_{(2,3,2)}$ and  
also in appending $\lambda^t$ to $D$ and each of $\Delta_{(3,2,1,1)}$, $\Delta_{(3,1,3)}$, and $\Delta_{(2,3,2)}$.
Therefore, we are interested in the following diagrams.

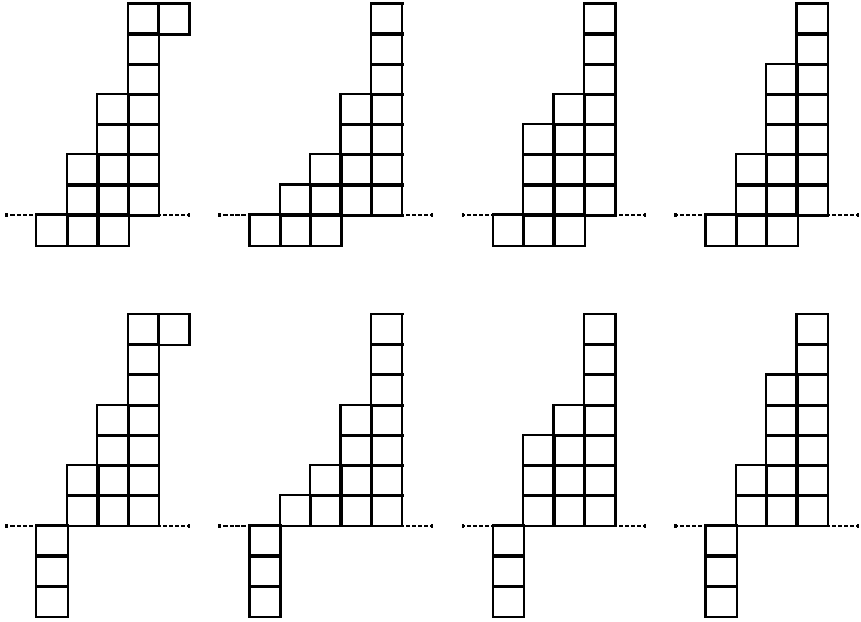
\setlength{\unitlength}{0.4mm}
\begin{picture}(100,60)(-60,0)

\put(-30,40){\framebox(10,10)[tl]{ }}
\put(-20,40){\framebox(10,10)[tl]{ }}

\put(-30,30){\framebox(10,10)[tl]{ }}

\put(-30,20){\framebox(10,10)[tl]{ }}

\put(-30,10){\framebox(10,10)[tl]{ }}
\put(-40,10){\framebox(10,10)[tl]{ }}

\put(-40,0){\framebox(10,10)[tl]{ }}
\put(-30,0){\framebox(10,10)[tl]{ }}

\put(-50,-10){\framebox(10,10)[tl]{ }}
\put(-40,-10){\framebox(10,10)[tl]{ }}
\put(-30,-10){\framebox(10,10)[tl]{ }}

\put(-50,-20){\framebox(10,10)[tl]{ }}
\put(-40,-20){\framebox(10,10)[tl]{ }}
\put(-30,-20){\framebox(10,10)[tl]{ }}

\put(-70,-20){\dashbox{1}(60,00)[tl]{ }}

\put(-60,-30){\framebox(10,10)[tl]{ }}
\put(-50,-30){\framebox(10,10)[tl]{ }}
\put(-40,-30){\framebox(10,10)[tl]{ }}

\put(50,40){\framebox(10,10)[tl]{ }}
\put(50,30){\framebox(10,10)[tl]{ }}
\put(50,20){\framebox(10,10)[tl]{ }}
\put(50,10){\framebox(10,10)[tl]{ }}
\put(50,0){\framebox(10,10)[tl]{ }}
\put(50,-10){\framebox(10,10)[tl]{ }}
\put(50,-20){\framebox(10,10)[tl]{ }}

\put(40,10){\framebox(10,10)[tl]{ }}
\put(40,0){\framebox(10,10)[tl]{ }}
\put(40,-10){\framebox(10,10)[tl]{ }}
\put(40,-20){\framebox(10,10)[tl]{ }}

\put(30,-10){\framebox(10,10)[tl]{ }}
\put(30,-20){\framebox(10,10)[tl]{ }}

\put(20,-20){\framebox(10,10)[tl]{ }}

\put(0,-20){\dashbox{1}(70,00)[tl]{ }}

\put(80,-20){\dashbox{1}(60,00)[tl]{ }}
\put(150,-20){\dashbox{1}(60,00)[tl]{ }}

\put(90,-30){\framebox(10,10)[tl]{ }}
\put(100,-30){\framebox(10,10)[tl]{ }}
\put(110,-30){\framebox(10,10)[tl]{ }}

\put(10,-30){\framebox(10,10)[tl]{ }}
\put(20,-30){\framebox(10,10)[tl]{ }}
\put(30,-30){\framebox(10,10)[tl]{ }}

\put(160,-30){\framebox(10,10)[tl]{ }}
\put(170,-30){\framebox(10,10)[tl]{ }}
\put(180,-30){\framebox(10,10)[tl]{ }}

\put(120,40){\framebox(10,10)[tl]{ }}
\put(120,30){\framebox(10,10)[tl]{ }}
\put(120,20){\framebox(10,10)[tl]{ }}
\put(120,10){\framebox(10,10)[tl]{ }}
\put(120,0){\framebox(10,10)[tl]{ }}
\put(120,-10){\framebox(10,10)[tl]{ }}
\put(120,-20){\framebox(10,10)[tl]{ }}

\put(110,10){\framebox(10,10)[tl]{ }}
\put(110,0){\framebox(10,10)[tl]{ }}
\put(110,-10){\framebox(10,10)[tl]{ }}
\put(110,-20){\framebox(10,10)[tl]{ }}

\put(100,0){\framebox(10,10)[tl]{ }}
\put(100,-10){\framebox(10,10)[tl]{ }}
\put(100,-20){\framebox(10,10)[tl]{ }}

\put(190,40){\framebox(10,10)[tl]{ }}
\put(190,30){\framebox(10,10)[tl]{ }}
\put(190,20){\framebox(10,10)[tl]{ }}
\put(190,10){\framebox(10,10)[tl]{ }}
\put(190,0){\framebox(10,10)[tl]{ }}
\put(190,-10){\framebox(10,10)[tl]{ }}
\put(190,-20){\framebox(10,10)[tl]{ }}

\put(180,20){\framebox(10,10)[tl]{ }}
\put(180,10){\framebox(10,10)[tl]{ }}
\put(180,0){\framebox(10,10)[tl]{ }}
\put(180,-10){\framebox(10,10)[tl]{ }}
\put(180,-20){\framebox(10,10)[tl]{ }}

\put(170,-10){\framebox(10,10)[tl]{ }}
\put(170,-20){\framebox(10,10)[tl]{ }}

\end{picture}

\

\

\

\

\setlength{\unitlength}{0.4mm}
\begin{picture}(100,60)(-60,0)

\put(-30,40){\framebox(10,10)[tl]{ }}
\put(-20,40){\framebox(10,10)[tl]{ }}

\put(-30,30){\framebox(10,10)[tl]{ }}

\put(-30,20){\framebox(10,10)[tl]{ }}

\put(-30,10){\framebox(10,10)[tl]{ }}
\put(-40,10){\framebox(10,10)[tl]{ }}

\put(-40,0){\framebox(10,10)[tl]{ }}
\put(-30,0){\framebox(10,10)[tl]{ }}

\put(-50,-10){\framebox(10,10)[tl]{ }}
\put(-40,-10){\framebox(10,10)[tl]{ }}
\put(-30,-10){\framebox(10,10)[tl]{ }}

\put(-50,-20){\framebox(10,10)[tl]{ }}
\put(-40,-20){\framebox(10,10)[tl]{ }}
\put(-30,-20){\framebox(10,10)[tl]{ }}

\put(-70,-20){\dashbox{1}(60,00)[tl]{ }}

\put(-60,-30){\framebox(10,10)[tl]{ }}
\put(-60,-40){\framebox(10,10)[tl]{ }}
\put(-60,-50){\framebox(10,10)[tl]{ }}

\put(50,40){\framebox(10,10)[tl]{ }}
\put(50,30){\framebox(10,10)[tl]{ }}
\put(50,20){\framebox(10,10)[tl]{ }}
\put(50,10){\framebox(10,10)[tl]{ }}
\put(50,0){\framebox(10,10)[tl]{ }}
\put(50,-10){\framebox(10,10)[tl]{ }}
\put(50,-20){\framebox(10,10)[tl]{ }}

\put(40,10){\framebox(10,10)[tl]{ }}
\put(40,0){\framebox(10,10)[tl]{ }}
\put(40,-10){\framebox(10,10)[tl]{ }}
\put(40,-20){\framebox(10,10)[tl]{ }}

\put(30,-10){\framebox(10,10)[tl]{ }}
\put(30,-20){\framebox(10,10)[tl]{ }}

\put(20,-20){\framebox(10,10)[tl]{ }}

\put(0,-20){\dashbox{1}(70,00)[tl]{ }}

\put(80,-20){\dashbox{1}(60,00)[tl]{ }}
\put(150,-20){\dashbox{1}(60,00)[tl]{ }}

\put(90,-30){\framebox(10,10)[tl]{ }}
\put(90,-40){\framebox(10,10)[tl]{ }}
\put(90,-50){\framebox(10,10)[tl]{ }}

\put(10,-30){\framebox(10,10)[tl]{ }}
\put(10,-40){\framebox(10,10)[tl]{ }}
\put(10,-50){\framebox(10,10)[tl]{ }}

\put(160,-30){\framebox(10,10)[tl]{ }}
\put(160,-40){\framebox(10,10)[tl]{ }}
\put(160,-50){\framebox(10,10)[tl]{ }}

\put(120,40){\framebox(10,10)[tl]{ }}
\put(120,30){\framebox(10,10)[tl]{ }}
\put(120,20){\framebox(10,10)[tl]{ }}
\put(120,10){\framebox(10,10)[tl]{ }}
\put(120,0){\framebox(10,10)[tl]{ }}
\put(120,-10){\framebox(10,10)[tl]{ }}
\put(120,-20){\framebox(10,10)[tl]{ }}

\put(110,10){\framebox(10,10)[tl]{ }}
\put(110,0){\framebox(10,10)[tl]{ }}
\put(110,-10){\framebox(10,10)[tl]{ }}
\put(110,-20){\framebox(10,10)[tl]{ }}

\put(100,0){\framebox(10,10)[tl]{ }}
\put(100,-10){\framebox(10,10)[tl]{ }}
\put(100,-20){\framebox(10,10)[tl]{ }}

\put(190,40){\framebox(10,10)[tl]{ }}
\put(190,30){\framebox(10,10)[tl]{ }}
\put(190,20){\framebox(10,10)[tl]{ }}
\put(190,10){\framebox(10,10)[tl]{ }}
\put(190,0){\framebox(10,10)[tl]{ }}
\put(190,-10){\framebox(10,10)[tl]{ }}
\put(190,-20){\framebox(10,10)[tl]{ }}

\put(180,20){\framebox(10,10)[tl]{ }}
\put(180,10){\framebox(10,10)[tl]{ }}
\put(180,0){\framebox(10,10)[tl]{ }}
\put(180,-10){\framebox(10,10)[tl]{ }}
\put(180,-20){\framebox(10,10)[tl]{ }}

\put(170,-10){\framebox(10,10)[tl]{ }}
\put(170,-20){\framebox(10,10)[tl]{ }}

\end{picture}

\

\

\

\

\

\

Theorem~\ref{sumofdiff} states that 
\[s_{\mathcal{S}(\lambda^t, D;1)} - s_{\mathcal{S}(\lambda, D;1)} \geq_s \sum_{\nu} c_{\kappa \nu}^{\rho}  (s_{\mathcal{S}(\lambda^t,  \alpha(\nu)^{\triangleleft};1)}- s_{\mathcal{S}(\lambda,  \alpha(\nu)^{\triangleleft};1)}) \geq_s 0. \]\\
In fact if we compute both

\begin{eqnarray*}
\sum_{\nu} c_{\kappa \nu}^{\rho}  (s_{\mathcal{S}(\lambda^t,  \alpha(\nu)^{\triangleleft};1)}- s_{\mathcal{S}(\lambda,  \alpha(\nu)^{\triangleleft};1)}) 
&=&  1(s_{\mathcal{S}(\lambda^t,  (3,2,1,1)^{\triangleleft};1)}- s_{\mathcal{S}(\lambda,  (3,2,1,1)^{\triangleleft};1)}) \\
&+&  1(s_{\mathcal{S}(\lambda^t,  (3,1,3)^{\triangleleft};1)}- s_{\mathcal{S}(\lambda,  (3,1,3)^{\triangleleft};1)}) \\
&+&  1(s_{\mathcal{S}(\lambda^t,  (2,3,2)^{\triangleleft};1)}- s_{\mathcal{S}(\lambda,  (2,3,2)^{\triangleleft};1)}) \\
\end{eqnarray*}
and 
\[s_{\mathcal{S}(\lambda^t, D;1)} - s_{\mathcal{S}(\lambda, D;1)},\]
then the difference is given by
\begin{eqnarray*}
&& \left( s_{\mathcal{S}(\lambda^t, D;1)} - s_{\mathcal{S}(\lambda, D;1)} \right) - \left(  \sum_{\nu} c_{\kappa \nu}^{\rho}  (s_{\mathcal{S}(\lambda^t,  \alpha(\nu)^{\triangleleft};1)}- s_{\mathcal{S}(\lambda,  \alpha(\nu)^{\triangleleft};1)})   \right) \\
 &=& s_{(5,4,3,2,1,1,1)} + s_{(5,4,2,2,2,1,1)}+s_{(5,4,2,2,1,1,1,1)}. \\
\end{eqnarray*}
Further, one can also check that 
\[\sum_{\nu} c_{\mu \nu}^{\rho}  (s_{\mathcal{S}(\lambda^t,  \alpha(\nu)^{\triangleleft};1)}- s_{\mathcal{S}(\lambda,  \alpha(\nu)^{\triangleleft};1)}) \geq_s 0. \]

\end{example}

\end{document}